\documentclass[a4paper, reqno, 11pt, oneside]{amsart}

\usepackage[all]{xy}
\xyoption{poly}
\usepackage{amssymb}
\usepackage{amstext}
\usepackage{amsmath}
\usepackage{amscd}
\usepackage{amsthm}
\usepackage{amsfonts}
\usepackage{graphicx}
\usepackage{latexsym}
\usepackage{mathrsfs}
\usepackage{nicematrix}

\usepackage[includehead, includefoot, top=25truemm, bottom=25truemm, left=25truemm, right=25truemm]{geometry}
\usepackage{xcolor}
\usepackage[colorlinks=true]{hyperref}
\definecolor{darkblue}{rgb}{0.0, 0.0, 0.55}
\definecolor{darkmagenta}{rgb}{0.55, 0.0, 0.55}
\definecolor{darkcyan}{rgb}{0.0, 0.55, 0.55}
\hypersetup{
	linkcolor  = darkblue,
	citecolor  = darkmagenta,
	urlcolor   = darkcyan,
	colorlinks = true,
}

\makeatletter
  
  \@addtoreset{equation}{section}
\makeatother

\makeatletter
\@namedef{subjclassname@2020}{%
  \textup{2020} Mathematics Subject Classification}
\makeatother

\newcommand{\calF}{\mathcal{F}}

\newcommand{\calV}{\mathcal{V}}

\newcommand{\NN}{\mathbb{N}}
\newcommand{\ZZ}{\mathbb{Z}}

\newcommand{\PP}{\mathbb{P}}

\def\opn#1#2{\def#1{\operatorname{#2}}} % to make operators
\opn\rank{rank} \opn\mnull{null} \opn\Iso{Iso} \opn\Sw{Sw}
\opn\type{type} \opn\outdeg{outdeg} \opn\indeg{indeg} 
\opn\Hom{Hom}

\def\GrMod{\operatorname{\mathsf{GrMod}}}
\def\GrAut{\operatorname{GrAut}}
\def\M{\operatorname{M}}
\def\Alt{\operatorname{Alt}}
\def\Im{\operatorname{Im}}
\def\Ker{\operatorname{Ker}}

\def\rnum#1{\expandafter{\romannumeral #1}}
\def\Rnum#1{\uppercase\expandafter{\romannumeral #1}}
\newcommand{\longsquiggly}{\xymatrix{{}\ar@{~>}[r]&{}}}

\newtheorem{thm}{Theorem}[section]
\newtheorem{cor}[thm]{Corollary}
\newtheorem{lem}[thm]{Lemma}
\newtheorem{prop}[thm]{Proposition}

\theoremstyle{definition}
\newtheorem{dfn}[thm]{Definition}
\newtheorem{ex}[thm]{Example}

\newtheorem{rem}[thm]{Remark}

\begin{document}

\title [Graded module categories over skew polynomial algebras at roots of unity]
{Combinatorics of graded module categories over skew polynomial algebras at roots of unity}

%first author
\author{Akihiro Higashitani}
\address{Department of Pure and Applied Mathematics, 
Graduate School of Information Science and Technology, 
Osaka University, 
1-5, Yamadaoka, Suita, Osaka 565-0871, Japan}
\email{higashitani@ist.osaka-u.ac.jp}

%second author
\author{Kenta Ueyama}
\address{Department of Mathematics,
Faculty of Science,
Shinshu University,
3-1-1 Asahi, Matsumoto, Nagano 390-8621, Japan}
\email{ueyama@shinshu-u.ac.jp}


%16W50: Graded rings and modules (associative rings and algebras)
%16S38: Rings arising from noncommutative algebraic geometry
%05C50: Graphs and linear algebra (matrices, eigenvalues, etc.)
%05E45: Combinatorial aspects of simplicial complexes
\begin{abstract}
We introduce an operation on skew-symmetric matrices over $\mathbb{Z}/\ell\mathbb{Z}$ called switching, and also define a class of skew-symmetric matrices over $\mathbb{Z}/\ell\mathbb{Z}$ referred to as modular Eulerian matrices. We then show that these are closely related to the graded module categories over skew polynomial algebras at $\ell$-th roots of unity.
As an application, we study the point simplicial complexes of skew polynomial algebras at cube roots of unity.
\end{abstract}

\maketitle 

\section{Introduction}
Classifying noetherian Koszul Artin-Schelter regular algebras up to equivalence of graded module categories is an important subject in noncommutative algebraic geometry. This is because it is closely related to classifying noncommutative projective spaces (see, e.g., \cite{AOU, AZ, ATV2, CGt, CG, IM, Ma, Mss, MU, NVZ, Si, Vi}). 
In this paper, we focus on a special class of Koszul AS-regular algebras referred to as \emph{standard graded skew polynomial algebras at roots of unity}, and we investigate the equivalence classes of their graded module categories using combinatorial methods. 

The \emph{point variety} $\Gamma_S$ of a skew polynomial algebra $S$ is known as a crucial

 geometric invariant under the equivalence of the graded module categories (see, e.g., \cite{ATV, BDL, Mss, R, Vi}).
Using the fact that the point variety $\Gamma_S$ of a skew polynomial algebra $S$ in $n$ variables is a union of linear subspaces of $\PP^{n-1}$, we can define a simplicial complex $\Delta_S$, which we call the \emph{point simplicial complex} of $S$, as a combinatorial aspect of $\Gamma_S$.

In previous work \cite{HU}, the authors studied \emph{standard graded ($\pm 1$)-skew polynomial algebras}, or in other words, standard graded skew polynomial algebras at square roots of unity.
They associated a graph $G_S$ with each standard graded ($\pm 1$)-skew polynomial algebra $S$, and proved the following theorem.

\begin{thm}[{\cite[Theorem 4.2]{HU}}]\label{thm:pm1}
Let $S$ and $S'$ be standard graded ($\pm 1$)-skew polynomial algebras (i.e., standard graded skew polynomial algebras at square roots of unity).
Then the following are equivalent.
\begin{enumerate}
\renewcommand{\labelenumi}{(\roman{enumi})}
\item The categories of graded right modules over $S$ and $S'$ are equivalent.
\item The graphs $G_{S}$ and $G_{S'}$ are switching equivalent.
\item The point simplicial complexes $\Delta_{S}$ and $\Delta_{S'}$ are isomorphic.
\end{enumerate}
\end{thm}

Now let us explain our main results. Let $\Alt_n(\ZZ/\ell \ZZ)$ be the set of $n\times n$ skew-symmetric matrices over $\ZZ/\ell\ZZ$.
In this study, we associate a skew-symmetric matrix $M_{S} \in \Alt_n(\ZZ/\ell \ZZ)$ with each standard graded skew polynomial algebra $S$ at $\ell$-th roots of unity. Moreover, by introducing an operation on matrices in $\Alt_n(\ZZ/\ell \ZZ)$, called \emph{switching}, we show the following combinatorial result, which recovers $\textnormal{(\rnum{1})}\Leftrightarrow \textnormal{(\rnum{2})}$ of Theorem \ref{thm:pm1}
when $\ell=2$.

\begin{thm}[Theorem \ref{thm.grmodl}]\label{thm:introsw}
Let $S$ and $S'$ be standard graded skew polynomial algebras at $\ell$-th roots of unity.
Then the following are equivalent.
\begin{enumerate}
\item The categories of graded right modules over $S$ and $S'$ are equivalent.
\item The matrices $M_{S}$ and $M_{S'}$ are switching equivalent.
\end{enumerate}
\end{thm}

It is known that the switching equivalence classes of graphs are deeply related to the isomorphism classes of Eulerian graphs, i.e., unlabeled Eulerian graphs (see \cite{MS, S}). As a generalization of Eulerian graphs, 
we introduce the notion of a matrix in $\Alt_n(\ZZ/\ell\ZZ)$ being a \emph{modular Eulerian matrix}, and show the following theorem.

\begin{thm}[{Theorem \ref{thm.Eulermodl}}]\label{thm:introEuler}
If $n$ is coprime to $\ell$, then each switching equivalence class in $\Alt_n(\ZZ/\ell \ZZ)$ contains exactly one isomorphism class of modular Eulerian matrices.
\end{thm}

As an application of Theorems \ref{thm:introsw} and \ref{thm:introEuler}, we give the following result on skew polynomial algebras at cube roots of unity; compare with 
Theorem \ref{thm:pm1} (which is for $\ell=2$) and Example \ref{ex:swsc} (which is for $\ell\geq 4$).

\begin{thm}[{Theorem \ref{thm.classps}, Example \ref{ex.simpcomp}(1)}]\label{thm:introps}
Let $S$ and $S'$ be standard graded skew polynomial algebras in $n$ variables at cube roots of unity.
If $n \leq 5$, then the following are equivalent.
\begin{enumerate}
\item The categories of graded right modules over $S$ and $S'$ are equivalent.
\item The point simplicial complexes $\Delta_{S}$ and $\Delta_{S'}$ are isomorphic.
\end{enumerate}
If $n=6$, then there exists a counterexample to $(2) \Rightarrow (1)$.
\end{thm}

At the end of this paper, we also examine the point simplicial complex $\Delta_S$ of a skew polynomial algebra $S$ at cube roots of unity from the perspective of isolating a vertex of the digraph associated with $S$ by switching (Propositions \ref{prop:isol1} and \ref{prop:isol2}).

\section{Preliminaries on skew polynomial algebras}

In this section, we discuss preliminaries on skew polynomial algebras.
Throughout this paper, let $k$ be an algebraically closed field of characteristic $0$.
Let $n$ be a positive integer and let $[n]=\{1,2,\dots,n\}$.
\begin{dfn}
A \emph{standard graded skew polynomial algebra} in $n$ variables is a graded algebra
\[ S_{\alpha} = k\langle x_1, \dots, x_n \rangle /(x_ix_j -\alpha_{ij} x_jx_i \mid 1\leq i,j \leq n) \]
where $\alpha =(\alpha_{ij}) \in \M_n(k)$ is a matrix such that $\alpha_{ii}=\alpha_{ij}\alpha_{ji}=1$ for all $1\leq i,j \leq n$, and $\deg x_i=1$ for all $1\leq i \leq n$.
\end{dfn}

Since a standard graded skew polynomial algebra in $n$ variables is an $n$-iterated Ore extension of $k$,
we see that it is a noetherian Koszul Artin-Schelter regular algebra of global dimension $n$ with Hilbert series $(1-t)^{-n}$ (see, e.g., \cite{R}).

\begin{dfn}
For two $n \times n$ matrices $M=(m_{ij}), M'=(m_{ij}')$, we say that $M$ and $M'$ are \emph{isomorphic} and write $M\cong M'$ if there exists a permutation $\sigma \in {\mathfrak S}_n$ such that $m'_{\sigma(i)\sigma(j)}=m_{ij}$ for $1 \leq i, j \leq n$. 
\end{dfn}

\begin{prop}[{\cite[Lemma 2.3]{Vi}, \cite[Theorem 2.4]{G}}]\label{prop.skewiso}
Let $S_{\alpha}$ and $S_{\alpha'}$ be standard graded skew polynomial algebras in $n$ variables.
Then $S_{\alpha} \cong S_{\alpha'}$ as graded algebras if and only if $\alpha \cong \alpha'$.
\end{prop}

Let $S_{\alpha}$ be a standard graded skew polynomial algebra in $n$ variables. 
We denote by $\GrMod S_{\alpha}$ the category of $\ZZ$-graded right $S_{\alpha}$-modules and degree-preserving $S_{\alpha}$-module homomorphisms.

\begin{dfn}\label{def:omega}
Let $\alpha=(\alpha_{ij}) \in \M_n(k)$ be a matrix with $\alpha_{ii}=\alpha_{ij}\alpha_{ji}=1$.
For $1\leq v \leq n$ and $\lambda \in k^{\times}$, let 
\[Y_{v,\lambda}=\begin{pNiceMatrix}[small, first-row, last-col,]
&&&&\text{\tiny $v$-th}&&&&\\
&&\hspace*{4mm}&&\lambda &&\hspace*{4mm}&&\\
&&1&&\Vdots&&1&&\\
&&&&&&&&\\
&&&&\lambda&&&&\\
\lambda^{-1}&\Cdots&&\lambda^{-1}&1& \lambda^{-1} &\Cdots&\lambda^{-1}&\text{\tiny $v$-th}\\
&&&&\lambda &&&&\\
&&1&&\Vdots&&1&&\\
&&&&\lambda&&&&
\end{pNiceMatrix}\quad \in \M_n(k).\]
We define $\mu_{v,\lambda}(\alpha)=(\beta_{ij}) \in \M_n(k)$ by
\begin{align*}
&\beta_{vi}=\lambda^{-1}\alpha_{vi} \quad \text{if} \; i\neq v,\\
&\beta_{iv}=\beta_{vi}^{-1}=\lambda\alpha_{iv} \quad \text{if} \; i\neq v,\\
&\beta_{ij}=\alpha_{ij} \quad \text{otherwise}.
\end{align*}
That is, \[\mu_{v,\lambda} (\alpha)=\alpha \circ Y_{v,\lambda},\] 
where $\circ$ means the Hadamard product (entrywise product) of matrices. 
We simply write $\mu_{v',\lambda'}\mu_{v,\lambda}(\alpha)$
for $\mu_{v',\lambda'}(\mu_{v,\lambda}(\alpha))$.
\end{dfn}

For an $\NN$-graded algebra $A=\bigoplus_{i \in \NN} A_i$ and a graded algebra automorphism $\theta \in \GrAut A$, the \emph{twisted graded algebra} of $A$ by $\theta$ is the graded algebra $A^{\theta}$ defined as follows:
$A^{\theta} =A$ as graded vector space, and the multiplication of $A^{\theta}$ is given by
\[a*b = a\theta^i(b)\]
for $a \in A_i, b \in A_j$.
By \cite[Theorem 3.1]{Z}, we see that $\GrMod A$ is equivalent to $\GrMod A^{\theta}$.

\begin{lem} \label{lem.swgen}
Let $S_{\alpha}$ and $S_{\alpha'}$ be standard graded skew polynomial algebras in $n$ variables.
If $\alpha'=\mu_{v,\lambda}(\alpha)$ for some $1\leq v\leq n$ and $\lambda \in k^{\times}$, then $\GrMod S_{\alpha}$ is equivalent to $\GrMod S_{\alpha'}$.  
\end{lem}

\begin{proof}
There exists a graded algebra automorphism $\theta \in \GrAut S_{\alpha}$ defined by $\theta(x_v)=\lambda x_v$ and $\theta(x_i)=x_i$ for $i \neq v$. In $(S_{\alpha})^{\theta}$, we have
\[x_v*x_i= x_v\theta(x_i)= x_vx_i= \alpha_{vi} x_ix_v=\lambda^{-1}\alpha_{vi}x_i(\lambda x_v)=\lambda^{-1}\alpha_{vi}x_i\theta(x_v) =\lambda^{-1}\alpha_{vi}x_i*x_v\]
for $i \neq v$ and $x_i*x_j=\alpha_{ij} x_j*x_i$ for $i\neq v,j\neq v$, so it follows that $(S_{\alpha})^{\theta}\cong S_{\alpha'}$. Therefore, $\GrMod S_{\alpha}$ is equivalent to $\GrMod S_{\alpha'}$ by \cite[Theorem 3.1]{Z}.
\end{proof}

\begin{thm}\label{thm.grmodgen}
	Let $S_{\alpha}$ and $S_{\alpha'}$ be standard graded skew polynomial algebras in $n$ variables. Then the following are equivalent.
	\begin{enumerate}
		\item $\GrMod S_{\alpha}$ is equivalent to $\GrMod S_{\alpha'}$.
		\item There exist a permutation $\sigma \in {\mathfrak S}_n$ and $\lambda_1,\dots,\lambda_n \in k^{\times}$ such that $\alpha_{\sigma(i)\sigma(j)}'=\lambda_i^{-1}\lambda_j\alpha_{ij}$ for all $1 \leq i<j\leq n$.
		\item There exists a permutation $\sigma \in {\mathfrak S}_n$ such that $\alpha_{\sigma(i)\sigma(j)}'\alpha_{\sigma(j)\sigma(h)}'\alpha_{\sigma(h)\sigma(i)}'=\alpha_{ij}\alpha_{jh}\alpha_{hi}$ for all $1\leq i<j<h\leq n$.
		\item There exist $\lambda_1,\dots,\lambda_n \in k^{\times}$ such that $\alpha'\cong \mu_{1,\lambda_1}\cdots\mu_{n,\lambda_n}(\alpha)$.
	\end{enumerate}
\end{thm}

\begin{proof}
$(4) \Rightarrow (1)$ follows from Proposition \ref{prop.skewiso} and Lemma \ref{lem.swgen}. 
$(1) \Rightarrow (2)$ was shown in \cite[Proposition 2.3]{Vi}.
$(2) \Rightarrow (3)$ is clear. Thus we show that $(3) \Rightarrow (4)$. 
	
For each $2 \leq i \leq n$, we define
$$\displaystyle \lambda_i:=\frac{\alpha_{\sigma(1)\sigma(i)}'}{\alpha_{1i}} \quad \text{ and }\quad \widetilde{\alpha}=(\widetilde{\alpha}_{ij}):=\mu_{2,\lambda_2} \cdots \mu_{n,\lambda_n}(\alpha).$$
In what follows, we prove that $\sigma$ induces an isomorphism of $\widetilde{\alpha}$ and $\alpha'$. 
What we have to do is to show that $\widetilde{\alpha}_{ab}=\alpha_{\sigma(a)\sigma(b)}'$ for any $1 \leq a<b \leq n$. 
	
Let $2 \leq i \leq n$. Then $\alpha_{\sigma(1)\sigma(i)}'=\lambda_i\alpha_{1i}$ by definition of $\lambda_i$. 
On the other hand, regarding $\widetilde{\alpha}$, since we apply $\mu_{i,\lambda_i}$ to $\alpha$, we see that $\widetilde{\alpha}_{1i}=\lambda_i\alpha_{1i}$ by definition of $\mu_{i,\lambda_i}$. 
Therefore, $\widetilde{\alpha}_{1i}=\alpha_{\sigma(1)\sigma(i)}'$. 
	
Let $2 \leq i<j \leq n$. Then we have $\alpha_{\sigma(1)\sigma(i)}'=\lambda_i\alpha_{1i}$ and $\alpha_{\sigma(1)\sigma(j)}'=\lambda_j\alpha_{1j}$. 
Moreover, we have
$\alpha_{\sigma(1)\sigma(i)}'\alpha_{\sigma(i)\sigma(j)}'\alpha_{\sigma(j)\sigma(1)}'=\alpha_{1i}\alpha_{ij}\alpha_{j1}$ by assumption.
Thus we obtain
$\alpha_{1i}\alpha_{ij}\alpha_{1j}^{-1}=\lambda_i\alpha_{1i}\alpha_{\sigma(i)\sigma(j)}'(\lambda_j\alpha_{1j})^{-1}$, and hence 
$\alpha_{\sigma(i)\sigma(j)}'=\lambda_i^{-1}\lambda_j\alpha_{ij}$.
On the other hand, since we apply $\mu_{i,\lambda_i}$ and $\mu_{j,\lambda_j}$ to $\alpha$, $\alpha_{ij}$ changes as follows: 
\[
\alpha_{ij} \;\overset{\mu_{i,\lambda_i}}{\longsquiggly}\; \lambda_i^{-1}\alpha_{ij} \;\overset{\mu_{j,\lambda_j}}{\longsquiggly}\;\lambda_j\lambda_i^{-1}\alpha_{ij}=\widetilde{\alpha}_{ij}.
\]
This shows $\widetilde{\alpha}_{ij}=\alpha_{\sigma(i)\sigma(j)}'$.

Therefore, we conclude that $\widetilde{\alpha}_{ab}=\alpha_{\sigma(a)\sigma(b)}'$ for any $1 \leq a<b \leq n$, as required.  
\end{proof}
\begin{rem}\label{rem:unique}
In the statement of Theorem~\ref{thm.grmodgen}(4), we note that the values of $\lambda_1,\ldots,\lambda_n \in k^\times$ are not unique. 
On the other hand, by definition of $\mu_{v,\lambda}$, we see that 
if $\alpha' \cong \mu_{1,\lambda_1} \cdots \mu_{n,\lambda_n}(\alpha)$, then $\alpha' \cong \alpha \circ (X_{1,\lambda_1} \circ \cdots \circ X_{n,\lambda_n})$. 
Moreover, if $\alpha' \cong \mu_{1,\lambda_1} \cdots \mu_{n,\lambda_n}(\alpha) \cong \mu_{1,\lambda_1'} \cdots \mu_{n,\lambda_n'}(\alpha)$, 
then $X_{1,\lambda_1} \circ \cdots \circ X_{n,\lambda_n} \cong X_{1,\lambda_1'} \circ \cdots \circ X_{n,\lambda_n'}$. 
This implies that the tuples $(\lambda_1,\ldots,\lambda_n)$ and $(\lambda_1',\ldots,\lambda_n')$ are equal up to a scalar product in $k^\times$ and up to permutation. 
\end{rem}

It should be noted that if $S$ and $S'$ are skew polynomial algebras with different numbers of variables, then $\GrMod S$ and $\GrMod S'$ are not equivalent by \cite[Proposition 5.7]{Z}.

The point schemes of Artin-Schelter regular algebras are a fundamental geometric object of the study in noncommutative algebraic geometry (see, e.g., \cite{R}).
Briefly speaking, the point scheme is defined by ``the parameter space of point modules'' (see \cite{ATV}).
Vitoria \cite{Vi} and independently Belmans, De~Laet, and Le~Bruyn \cite{BDL}
computed the point scheme of $S_\alpha$ explicitly.

For a subset $\{i_1,\dots,i_s\} \subset [n]$, we define the subspace
\[\PP(i_1,\dots,i_s):= \bigcap_{i \in [n]\setminus \{i_1,\dots,i_s\}} {\mathcal V}(x_i) \quad  \subset \PP^{n-1},\]
where ${\mathcal V}(x_i)$ means the zero-locus of $x_i$ in $\PP^{n-1}$.
%${\mathcal V}(x_i)=\{[x_1 : \cdots : x_n] \in \PP^{n-1} \mid x_i=0 \}.$

\begin{thm}[{\cite[Proposition 4.2]{Vi}, \cite[Theorem 1(1)]{BDL}}]\label{thm.ps}
Let $S_\alpha = k\langle x_1, \dots, x_n \rangle/(x_ix_j-\alpha_{ij}x_jx_i)$ be a standard graded skew polynomial algebra.
Then the point scheme of $S_\alpha$ is given by
\[ \Gamma_{\alpha}= \bigcap_{\substack{1\leq i<j<h\leq n \\ \alpha_{ij}\alpha_{jh}\alpha_{hi} \neq 1}}\calV(x_ix_jx_h) \quad \subset \PP^{n-1}. \]
In particular, it is the union of a collection of subspaces $\PP(i_1,\dots,i_s)$.
\end{thm}

\begin{rem}\label{rem:pv}
By Theorem \ref{thm.ps}, the point scheme $\Gamma_{\alpha}$ is reduced,
so it will henceforth be referred to as the \emph{point variety} of $S_{\alpha}$.
\end{rem}

\begin{ex}\label{ex:ps3}
Let $S_\alpha = k\langle x_1, x_2, x_3 \rangle/(x_ix_j-\alpha_{ij}x_jx_i)$ be a standard graded skew polynomial algebra in $3$ variables. Then the point variety $\Gamma_{\alpha}$ is given as follows:
\[
\Gamma_{\alpha}=
\begin{cases}
\PP^2 = \PP(1,2,3) & \text{if}\; \alpha_{12}\alpha_{23}\alpha_{31}=1,\\
\calV(x_1x_2x_3)= \calV(x_1)\cup\calV(x_2)\cup\calV(x_3)=\PP(2,3) \cup \PP(1,3) \cup \PP(1,2) & \text{if}\; \alpha_{12}\alpha_{23}\alpha_{31}\neq 1.
\end{cases}
\]
\end{ex}

The following is a well-known consequence of a general result due to Zhang \cite{Z}; see \cite[Example 4.1]{Mcp} for details.

\begin{prop}\label{prop.grps}
Let $S_{\alpha}$ and $S_{\alpha'}$ be standard graded skew polynomial algebras. If $\GrMod S_{\alpha}$ is equivalent to $\GrMod S_{\alpha'}$, then $\Gamma_{\alpha}$ is isomorphic to $\Gamma_{\alpha'}$.
\end{prop}

We say that a collection $\Delta$ of subsets of a finite set $V$ is an (abstract) \emph{simplicial complex} on the vertex set $V$ 
if the following conditions are satisfied:
\begin{itemize}
\item $\{v\} \in \Delta$ for each $v \in V$,
\item $F \in \Delta$ and $F' \subset F$ imply $F' \in \Delta$. 
\end{itemize}
Let $\Delta$ be a simplicial complex on a nonempty finite set $V$.
%, which is a family of subsets of $S$ closed under inclusion. 
%Assume that $\Delta$ contains at least one non-empty subset. 
The \emph{dimension} of $\Delta$ is defined by $\dim \Delta=\max\{|F|-1 \mid F \in \Delta\}$.
%where we let $\dim \Delta=-1$ \comment{if $\Delta=\{\emptyset\}$. 
Let $\calF(\Delta)$ denote the set of all facets (i.e., maximal faces) of $\Delta$. 
For two simplicial complexes $\Delta$ and $\Delta'$, 
we say that $\Delta$ and $\Delta'$ are \emph{isomorphic} as simplicial complexes 
if there exists a bijection $\phi$ between the vertex sets
$V(\Delta)$ and $V(\Delta')$ which induces a bijection between $\calF(\Delta)$ and $\calF(\Delta')$. 
We here define the simplicial complex associated with a standard graded skew polynomial algebra $S_\alpha$.

\begin{dfn}
Let $S_{\alpha}$ be a standard graded skew polynomial algebra in $n$ variables.
We define the simplicial complex $\Delta_{\alpha}$ on $[n]$ as follows:
a subset $F \subset [n]$ belongs to $\Delta_{\alpha}$ if and only if $\alpha_{ij}\alpha_{jh}\alpha_{hi}=1$ for all distinct $i,j,h$ such that $\{i,j,h\}\subset F$. We call $\Delta_{\alpha}$ the \emph{point simplicial complex}
of $S_{\alpha}$.
\end{dfn}

The point variety $\Gamma_{\alpha}$ is closely related to the point simplicial complex $\Delta_{\alpha}$. Indeed, we have $\Gamma_\alpha=\bigcup_{F \in \calF(\Delta_\alpha)}\PP(F)$, that is,
%\begin{align}\label{eq:ps}
%\end{align}
\begin{align}\label{eq:calFD}
\calF(\Delta_\alpha)=\{F \subset [n] \mid \PP(F) \ \text{is an irreducible component of $\Gamma_\alpha$}\}.
\end{align}
Moreover, the following holds.

\begin{prop}\label{prop.pssc}
Let $S_{\alpha}$ and $S_{\alpha'}$ be standard graded skew polynomial algebras. Then  $\Gamma_{\alpha}$ and $\Gamma_{\alpha'}$ are isomorphic if and only if $\Delta_{\alpha}$ and $\Delta_{\alpha'}$ are isomorphic.
\end{prop}

\begin{proof}
This follows from \cite[Lemma 2.14]{HU} and \eqref{eq:calFD}.
\end{proof}

\begin{ex}\label{ex:psimp3}
Let $S_\alpha = k\langle x_1, x_2, x_3 \rangle/(x_ix_j-\alpha_{ij}x_jx_i)$ be a standard graded skew polynomial algebra in $3$ variables. Then the point simplicial complex $\Delta_{\alpha}$ is given as follows:
\[
\calF(\Delta_{\alpha})=
\begin{cases}
\{\{1,2,3\}\} & \text{if}\; \alpha_{12}\alpha_{23}\alpha_{31}=1,\\
\{\{2,3\}, \{1,3\},\{1,2\}\}& \text{if}\; \alpha_{12}\alpha_{23}\alpha_{31}\neq 1.
\end{cases}
\]
\end{ex}

Let $S_{\alpha}$ and $S_{\alpha'}$ be standard graded skew polynomial algebras in $n$ variables. The following is a summary of this section:
\[
\xymatrix@R=1.5pc@C=0.3pc@M=5pt{
\txt{$\GrMod S_{\alpha}$ is equivalent to $\GrMod S_{\alpha'}$\\
(graded module categories)} \ar@{=>}[d]_-{\text{Prop.\,\ref{prop.grps}}}
&\overset{\text{Thm.\,\ref{thm.grmodgen}}}{\Longleftrightarrow}
&\txt{$\exists \lambda_1,\dots,\lambda_n \in k^{\times}\;\text{s.t.}\; \alpha'\cong \mu_{1,\lambda_1}\cdots\mu_{n,\lambda_n}(\alpha)$\\
(matrices)}\\
\txt{$\Gamma_{\alpha} \cong \Gamma_{\alpha'}$\\
(point varieties)} %\ar@{<=>}[r]
&\overset{\text{Prop.\,\ref{prop.pssc}}}{\Longleftrightarrow}
&
\txt{$\Delta_{\alpha'} \cong \Delta_{\alpha'}$\\
(point simplicial complexes)}.
}
\]

\section{skew polynomial algebras at roots of unity and switching}

Here and throughout the paper, let $\ell$ be an integer greater than or equal to $2$.
In this section, we establish a connection between ``the equivalence of graded module categories over skew polynomial algebras at $\ell$-th roots of unity'' and ``the switching equivalence of skew-symmetric matrices over $\ZZ/\ell \ZZ$''.

\begin{dfn}
A \emph{standard graded skew polynomial algebra at $\ell$-th roots of unity} is a standard graded skew polynomial algebra $S_{\omega}$ such that $\omega_{ij}$ are $\ell$-th roots of unity for all $1\leq i,j \leq n$.
\end{dfn}

From now on, we fix an $\ell$-th primitive root of unity $\zeta_\ell$.

\begin{dfn}
Let $S_{\omega}$ be a standard graded skew polynomial algebra at $\ell$-th roots of unity. Then $\omega_{ij}$ can be represented as $\zeta_\ell^{m_{ij}}$ for some $m_{ij} \in \ZZ/\ell\ZZ$.
The \emph{E-matrix $M_{\omega} \in \M_n(\ZZ/\ell\ZZ)$ associated with $S_{\omega}$} is defined by $M_{\omega}=(m_{ij})$. Note that $M_\omega$ is uniquely defined from $S_\omega$. 
\end{dfn}

Recall that a matrix $M \in \M_n(\ZZ/\ell\ZZ)$ is called a \emph{skew-symmetric matrix} if $m_{ij}+m_{ji}=m_{ii}=0$ in $\ZZ/\ell\ZZ$ for any $i,j$. 
Let 
\[
\Alt_n(\ZZ/\ell\ZZ)=\{M \in \M_n(\ZZ/\ell\ZZ) \mid \text{$M$ is a skew-symmetric matrix} \}.
\]
It is easy to see that $M_{\omega} \in \Alt_n(\ZZ/\ell\ZZ)$.  

\begin{prop}\label{prop.skewisol}
Let $S_{\omega}$ and $S_{\omega'}$ be standard graded skew polynomial algebras in $n$ variables at $\ell$-th roots of unity.
Then $S_{\omega} \cong S_{\omega'}$ as graded algebras if and only if $M_{\omega} \cong M_{\omega'}$.
\end{prop}

\begin{proof}
This follows from Proposition \ref{prop.skewiso}.
\end{proof}

\begin{dfn}\label{dfn:switchmodl}
(1) Let $M=(m_{ij}) \in \Alt_n(\ZZ/\ell\ZZ)$ be a skew-symmetric matrix. %such that $m_{ij}+m_{ji}=m_{ii}=0$. 
For $1\leq v \leq n$, let \[X_v=
\begin{pNiceMatrix}[small, first-row, last-col,]
&&&&\text{\tiny $v$-th}&&&&\\
&&\hspace*{4mm}&&1&&\hspace*{4mm}&&\\
&&&&\Vdots&&&&\\
&&&&&&&&\\
&&&&1&&&&\\
-1&\Cdots&&-1&0&-1&\Cdots&-1&\text{\tiny $v$-th}\\
&&&&1&&&&\\
&&&&\Vdots&&&&\\
&&&&1&&&&
\end{pNiceMatrix}\quad \in \Alt_n(\ZZ/\ell\ZZ)
\] 
We define a skew-symmetric matrix $\mu_v (M)=(n_{ij})$ by 
\begin{align*}
&n_{vi}=m_{vi}-1 \quad \text{if} \; i\neq v,\\
&n_{iv}=-n_{vi}=m_{iv}+1 \quad \text{if} \; i\neq v,\\
&n_{ij}=m_{ij} \quad \text{otherwise}.
\end{align*}
That is, 
\[
\mu_v (M)=M+X_v. 
\]
We call $\mu_v(M)$ the \emph{switching} of $M$ at $v$.
We simply write 
$\mu_{v'}\mu_{v}(M)$
for $\mu_{v'}(\mu_{v}(M))$
and
$\mu_{v}^2(M)$
for $\mu_{v}(\mu_{v}(M))$.

(2)  Let $M=(m_{ij}), M'=(m_{ij}') \in \Alt_n(\ZZ/\ell\ZZ)$ be skew-symmetric matrices.
We say that $M$ and $M'$ are \emph{switching equivalent} if there exist $i_1,\dots, i_n \in \ZZ_{>0}$ such that  
$M' \cong \mu_{1}^{i_1}\cdots\mu_{n}^{i_n}(M)$.
\end{dfn}

Note that $\mu_{v}^\ell(M)=M$, $\mu_{1}\cdots \mu_{n}(M)=M$, and $\mu_{v'}\mu_{v}(M)=\mu_{v}\mu_{v'}(M)$ hold. 
Note also that if $\omega=(\zeta_\ell^{m_{ij}}) \in \M_n(k)$, then $\mu_{v,\zeta_\ell}(\omega)$ in the sense of Definition~\ref{def:omega} is given by $(\zeta_\ell^{n_{ij}})$, where $n_{ij}$ are as in Definition \ref{dfn:switchmodl},
%so applying $\mu_v$ to $M$ corresponds to applying $\mu_{v,\zeta_\ell}$ to $\alpha$. 
i.e., $\mu_v(M_\omega)=M_{\mu_{v,\zeta_\ell}(\omega)}$. 

As a further remark, the switching equivalence of skew-symmetric matrices $M$ and $M'$ is defined by ``$M' \cong \mu_{1}^{i_1}\cdots\mu_{n}^{i_n}(M)$'', not by ``$M' = \mu_{1}^{i_1}\cdots\mu_{n}^{i_n}(M)$''.
This ensures that if two skew-symmetric matrices are isomorphic, then they are switching equivalent.

Throughout this paper, ``\emph{graph}'' means a finite undirected graph without loops and multiple edges. 
\begin{ex}\label{ex.2graph}
We consider the case $\ell=2$.
In this case, a skew-symmetric matrix $M=(m_{ij}) \in \Alt_n(\ZZ/2\ZZ)$ can be identified with the adjacency matrix of the graph $G$ with $V(G)=[n]$, 
where we recall that the \emph{adjacency matrix} $M(G')$ of a graph $G'$ with $V(G')=[n]$  is defined as follows: 
\begin{itemize}
\item $M(G')=(m_{ij}') \in \Alt_n(\ZZ/2\ZZ)$,
\item if there exists no edge between $i$ and $j$ in $G'$, then we define $m'_{ij}=0$,
\item if there exists an edge between $i$ and $j$ in $G'$, then we define $m'_{ij}=m'_{ji}=1$.
%\item the vertex set is $V(G)=[n]$,
%\item if $m_{ij}=0$, then there exists no edge between $i$ and $j$ in $G$,
%\item if $m_{ij}=1$, equivalently, $m_{ji}=1$, then there exists an edge between $i$ and $j$ in $G$.
\end{itemize}
Let $G(M)$ denote the graph whose adjacency matrix is $M$. Then $G(\mu_v(M))$ is given as follows:
\begin{itemize}
\item the vertex set is $V(G(\mu_v(M)))=V(G(M))=[n]$,
\item the edge set is $E(G(\mu_v(M)))=\{vi \mid vi \not \in E(G(M)), i \neq v\} \cup \{ij \mid ij\in E(G(M)), i\neq v, j\neq v\}$.
\end{itemize}
Therefore, the switching of $M$ at $v$ is the same notion as the switching of $G(M)$ at $v$, introduced by van~Lint and Seidel \cite{vLS}. For example, 
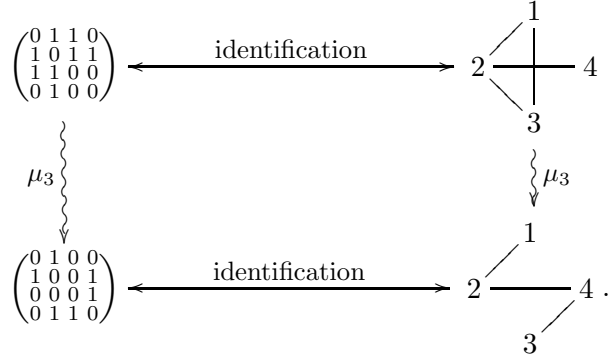
\[
\xymatrix@R=1.8pc@C=10pc{
{\begin{pNiceMatrix}[small]
0&1&1&0\\
1&0&1&1\\
1&1&0&0\\
0&1&0&0
\end{pNiceMatrix}}
\ar@{<->}[r]%^-{\txt{\small identification}}
\ar@{~>}[d]_-{\txt{\small $\mu_3$}}
&{\xy /r1.75pc/: 
{\xypolygon4{~={90}~*{\xypolynode}~>{}}},
"1";"2"**@{-},
"1";"3"**@{-},
"2";"3"**@{-},
"2";"4"**@{-},
\endxy}
\ar@{~>}[d]^-{\txt{\small $\mu_3$}}
\\
{\begin{pNiceMatrix}[small]
0&1&0&0\\
1&0&0&1\\
0&0&0&1\\
0&1&1&0
\end{pNiceMatrix}}
\ar@{<->}[r]%^-{\txt{\small identification}}
&{\xy /r1.75pc/: 
{\xypolygon4{~={90}~*{\xypolynode}~>{}}},
"1";"2"**@{-},
%"1";"3"**@{-},
%"2";"3"**@{-},
"3";"4"**@{-},
"2";"4"**@{-},
\endxy}.
}
\]
\end{ex}

Throughout this paper, ``\emph{digraph}'' means a finite directed graph without loops and multiple directed edges (not allowed opposite directions, either). 

\begin{ex}\label{ex.3graph}
We consider the case $\ell=3$.
In this case, a skew-symmetric matrix $M=(m_{ij}) \in \Alt_n(\ZZ/3\ZZ)$ can be identified with the adjacency matrix of the digraph $G$ with $V(G)=[n]$, where we note that the \emph{adjacency matrix} $M(G')$ of a digraph $G'$ with $V(G')=[n]$ is defined as follows: 
\begin{itemize}
\item $M(G')=(m_{ij}') \in \Alt_n(\ZZ/3\ZZ)$,
\item if there exists no directed edge between $i$ and $j$ in $G'$, then we define $m'_{ij}=0$,
\item if there exists a directed edge from $i$ to $j$ in $G'$, then we define $m'_{ij}=1$ and $m'_{ji}=2$.
%\item the vertex set is $V(G)=[n]$,
%\item if $m_{ij}=0$, then there exists no directed edge between $i$ and $j$ in $G$,
%\item if $m_{ij}=1$, equivalently, $m_{ji}=2$, then there exists a directed edge from $i$ to $j$ in $G$.
\end{itemize}
Let $G(M)$ denote the digraph whose adjacency matrix is $M$. Then $G(\mu_v(M))$ is given as follows: 
\begin{itemize}
\item the vertex set is $V(G(\mu_v(M)))=V(G(M))=[n]$,
\item the edge set is $E(G(\mu_v(M)))=\{vi \mid iv \in E(G(M)), i \neq v\} \cup \{iv \mid iv, vi \not \in E(G(M)), i \neq v\} \cup \{ij \mid ij\in E(G(M)), i\neq v, j\neq v\}$.
\end{itemize}
Therefore, the switching of $M$ at $v$ is the same notion as the switching of $G(M)$ at $v$, studied by Cheng and Wells Jr.\ \cite{CW}. For example,
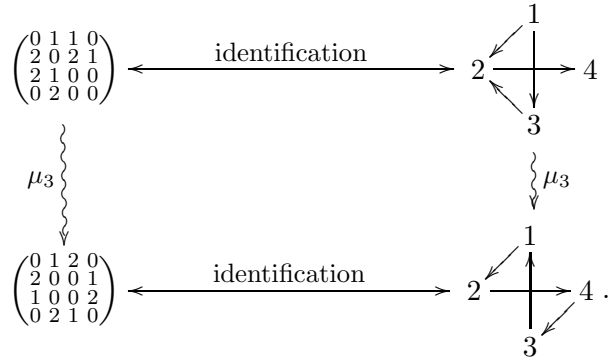
\[
\xymatrix@R=1.8pc@C=10pc{
{\begin{pNiceMatrix}[small]
0&1&1&0\\
2&0&2&1\\
2&1&0&0\\
0&2&0&0
\end{pNiceMatrix}}
\ar@{<->}[r]%^-{\txt{\small identification}}
\ar@{~>}[d]_-{\txt{\small $\mu_3$}}
&{\xy /r1.75pc/: 
{\xypolygon4{~={90}~*{\xypolynode}~>{}}},
\ar@{->}"3";"2",
\ar@{->}"1";"3",
\ar@{->}"1";"2",
\ar@{->}"2";"4",
\endxy
}
\ar@{~>}[d]^-{\txt{\small $\mu_3$}}
\\
{\begin{pNiceMatrix}[small]
0&1&2&0\\
2&0&0&1\\
1&0&0&2\\
0&2&1&0
\end{pNiceMatrix}}
\ar@{<->}[r]%^-{\txt{\small identification}}
&{\xy /r1.75pc/: 
{\xypolygon4{~={90}~*{\xypolynode}~>{}}},
\ar@{->}"3";"1",
\ar@{->}"1";"2",
\ar@{->}"2";"4",
\ar@{->}"4";"3",
\endxy
}.
}
\]
\end{ex}

We give the following observation, which will be used in Section 5. 
\begin{prop}\label{prop:iso}
For any $M \in \Alt_n(\ZZ/\ell\ZZ)$ and $v \in [n]$, there exists a unique $M' \in \Alt_n(\ZZ/\ell\ZZ)$ 
which is given by iterated switching applied to $M$ and whose $v$-th rows and columns are all $0$. 
\end{prop}
\begin{proof}
We may assume $v=1$ without loss of generality. 
Let $M=(m_{ij}) \in \Alt_n(\ZZ/\ell\ZZ)$. Then $M':=M+\sum_{i=2}^n m_{i1}X_i$ satisfies the required condition. 

In what follows, we prove the uniqueness of $M'$. 
Suppose that $M''=M+\sum_{i=1}^n c_iX_i$ also satisfies the condition. 
Note that $X_1+\cdots+X_n=O$, where $O$ stands for the zero matrix.  
Then \[M''-M'=c_1X_1+\sum_{i=2}^n (c_i-m_{i1})X_i=\sum_{i=2}^n (c_i-c_1-m_{i1})X_i.\]
Since the entries in the $1$st row of $M''-M'$ are all $0$, we have $c_i-c_1-m_{i1}=0$, i.e., $m_{i1}=c_i-c_1$ for $i=2,\ldots,n$. 
This implies $M'=M''$. 
\end{proof}

We now give the following combinatorial theorem.
\begin{thm} \label{thm.grmodl}
Let $S_{\omega}$ and $S_{\omega'}$ be standard graded skew polynomial algebras at $\ell$-th roots of unity,
and let $M_{\omega}$ and $M_{\omega'}$ be their associated E-matrices. 
Then $\GrMod S_{\omega}$ and $\GrMod S_{\omega'}$ are equivalent if and only if 
$M_{\omega}$ and $M_{\omega'}$ are switching equivalent.
\end{thm}

\begin{proof}
$(\Leftarrow)$ If $M_{\omega'} \cong \mu_{1}^{i_1}\cdots\mu_{n}^{i_n}(M_{\omega})$, then we see that
$\omega'\cong\mu_{1,\zeta_{\ell}^{i_1}}\cdots\mu_{n,\zeta_{\ell}^{i_n}}(\omega)$, so $\GrMod S_{\omega}$ is equivalent to $\GrMod S_{\omega'}$ by Theorem \ref{thm.grmodgen}.

$(\Rightarrow)$ If $\GrMod S_{\omega}$ is equivalent to $\GrMod S_{\omega'}$, then there exist $\lambda_1,\dots, \lambda_n \in k^\times$ such that $\omega'\cong \mu_{1,\lambda_1}\cdots\mu_{n,\lambda_n}(\omega)$ by Theorem \ref{thm.grmodgen}. Since $\omega_{ij}$ and $\omega'_{ij}$ are $\ell$-th roots of unity, it follows that $\lambda_i^{-1}\lambda_j$ are $\ell$-th roots of unity for all $1\leq i,j\leq n$. 
We now fix $s\in[n]$ and take $i_j\in\ZZ_{>0}$ to satisfy $\lambda_s^{-1}\lambda_j=\zeta_{\ell}^{i_j}$
for any $1\leq j\leq n$.
Since $\omega'\cong \mu_{1,\lambda_1}\cdots\mu_{n,\lambda_n}(\omega) = \mu_{1,\zeta_{\ell}^{i_1}}\cdots\mu_{n,\zeta_{\ell}^{i_n}}(\omega)$,
we have $M_{\omega'} \cong \mu_{1}^{i_1}\cdots\mu_{n}^{i_n}(M_{\omega})$.
(Note that, similar to Remark~\ref{rem:unique}, $i_1\dots,i_n$ are not unique.)
%it follows from Theorem \ref{thm.grmodgen} that 
%$\GrMod S_{\omega}$ is equivalent to $\GrMod S_{\omega'}$ if and only if there exist $i_1,\dots,i_n \in \ZZ_{>0}$ such that $\omega'\cong \mu_{1,\zeta_{\ell}^{i_1}}\cdots\mu_{n,\zeta_{\ell}^{i_n}}(\omega)$.
%Note that any $\lambda_1,\ldots,\lambda_n \in k^\times$ provided by Theorem~\ref{thm.grmodgen} gives a matrix isomorphic to $\omega'$ (see Remark~\ref{rem:unique}).
\end{proof}

\begin{rem}
Theorem \ref{thm.grmodl} is a generalization of $\textnormal{(1)}\Leftrightarrow \textnormal{(4)}$ of \cite[Theorem 4.2]{HU}.
(It should be noted that, for a skew polynomial algebra $S_\omega$ at square roots of unity, the graph associated with $S_\omega$ considered in \cite{HU} is the complement of our graph $M_\omega$; however, it is easy to see that two given graphs are switching equivalent if and only if their complement graphs are switching equivalent.)
\end{rem}

\begin{ex} \label{ex:swsc}
Assume that $\ell \geq 4$. Consider the standard graded skew polynomial algebras
$S_{\omega}$ and $S_{\omega'}$ in $3$ variables at $\ell$-th roots of unity with E-matrices
\[
M_{\omega}=
\begin{pmatrix}
0 &0 &0\\ 
0 &0 &1\\ 
0 &\ell-1 &0\\
\end{pmatrix}
\qquad\qquad
M_{\omega'}=
\begin{pmatrix}
0 &0 &0\\ 
0 &0 &2\\ 
0 &\ell-2 &0\\
\end{pmatrix}
\]
Then one can easily verify that  $M_{\omega}$ and $M_{\omega'}$ are not switching equivalent, so $\GrMod S_{\omega}$ and $\GrMod S_{\omega'}$ are not equivalent by Theorem \ref{thm.grmodl}.
On the other hand, we see that the point simplicial complexes $\Delta_{\omega}$ and $\Delta_{\omega'}$ are coincident by Example \ref{ex:psimp3}. Compare with 
Theorem \ref{thm:pm1} (which is for $\ell=2$) and Theorem \ref{thm.classps} later (which is for $\ell=3$).
\end{ex}

\section{Modular Eulerian matrices}
In this section, we introduce the notion of a matrix in $\Alt_n(\ZZ/\ell\ZZ)$ being a modular Eulerian matrix.
We then discuss the relationship between modular Eulerian matrices
and switching equivalence classes in $\Alt_n(\ZZ/\ell\ZZ)$.

As mentioned in Example \ref{ex.2graph} (resp.\ Example \ref{ex.3graph}), 
a skew-symmetric matrix $M \in \Alt_n(\ZZ/2 \ZZ)$ (resp.\ $M \in \Alt_n(\ZZ/3 \ZZ)$) can be regarded as the adjacency matrix of the graph  (resp.\ digraph) $G(M)$. 
%Thought this section, we frequently use these identifications.

\begin{dfn}
We say that a skew-symmetric matrix $M=(m_{ij}) \in \Alt_n(\ZZ/\ell \ZZ)$ is a \emph{modular Eulerian matrix} if $\sum_{j=1}^n m_{ij} =  0$ holds in $\ZZ/\ell\ZZ$ for every $i$. 
\end{dfn}

\begin{ex}
A matrix $M \in \Alt_n(\ZZ/2 \ZZ)$ is a modular Eulerian matrix if and only if the graph $G(M)$ is an Eulerian graph, i.e., a graph where every vertex has even degree. 
\end{ex}

\begin{ex}
We say that a digraph $G$ is a \emph{modular Eulerian digraph} if $\outdeg v \equiv \indeg v \pmod{3}$ holds for every vertex $v$ of $G$, where $\outdeg v$ (resp. $\indeg v$) denotes the outdegree (resp. indegree) of $v$, i.e., the number of vertices heading from (resp. heading to) $v$. 
A matrix $M \in \Alt_n(\ZZ/3 \ZZ)$ is a modular Eulerian matrix if and only if the digraph $G(M)$ is a modular Eulerian digraph. 
\end{ex}

Recall that a digraph is called \emph{Eulerian} if $\outdeg v =\indeg v$ holds for every vertex $v$. 
Notice that a modular Eulerian digraph is not necessarily an Eulerian digraph. For example, 
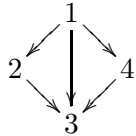
\[
\xy /r1.75pc/: 
{\xypolygon4{~={90}~*{\xypolynode}~>{}}},
\ar@{->}"1";"2",
\ar@{->}"1";"3",
\ar@{->}"1";"4",
\ar@{->}"2";"3",
\ar@{->}"4";"3",
\endxy
\]
is a modular Eulerian digraph, while this is not an Eulerian digraph.  

Seidel \cite{S} showed that every switching equivalence class of unlabeled graphs whose number of vertices is odd contains a unique unlabeled Eulerian graph. We generalize this result as follows.

\begin{thm}\label{thm.Eulermodl}
If $n$ is coprime to $\ell$, then each switching equivalence class in $\Alt_n(\ZZ/\ell \ZZ)$ contains exactly one isomorphism class of modular Eulerian matrices.
\end{thm}

\begin{proof} 
\noindent{\bf Existence}: 
Given $M=(m_{ij}) \in \Alt_n(\ZZ/\ell\ZZ)$, let $U_k=\{ i \in [n] \mid \sum_{j=1}^n m_{ij}= k\}$ and $u_k = |U_k|$ for $k=0,1,\ldots,\ell-1$. 
Then we see that $\sum_{k=0}^{\ell-1}ku_k=\sum_{1 \leq i,j \leq n}m_{ij} = 0$ in $\ZZ/\ell\ZZ$. 

Since we assume $n$ is coprime to $\ell$, there exists $s \in [\ell-1]$ such that $sn \equiv 1 \pmod{\ell}$. Let 
\[M'=\prod_{k=1}^{\ell-1}\prod_{v \in U_k}\mu_v^{sk}(M)=M+\sum_{k=1}^{\ell-1}\left( \sum_{v \in U_k}sk X_v \right),\]
where $\prod$ stands for the composition of certain $\mu_v$'s. (See Definition~\ref{dfn:switchmodl}(1) for the second equality.) 
Our goal is to show that $M'$ is a modular Eulerian matrix. 
We have
\begin{align}\label{eq:first}
\sum_{j=1}^n(M')_{uj}&=\sum_{j=1}^n\left(M+\sum_{k=1}^{\ell-1}\left( \sum_{v \in U_k}sk X_v \right)\right)_{uj} 
=\sum_{j=1}^nM_{uj}+\sum_{k=1}^{\ell-1}\left( \sum_{v \in U_k}sk \sum_{j=1}^n(X_v)_{uj} \right).
\end{align}
Here, by definition of $X_v$, we observe that 
\begin{align}\label{eq:observe}\sum_{j=1}^n(X_v)_{ij}=\begin{cases} 1-n &\text{ if }i=v, \\ 1 &\text{ if }i \neq v.  \end{cases}\end{align} 
Therefore, we see the following: if $u \in U_0$, then 
\begin{align*}
\sum_{j=1}^n(M')_{uj}
&\overset{\text{\eqref{eq:first}}}{=}\sum_{j=1}^nM_{uj}+\sum_{k=1}^{\ell-1}\left( \sum_{v \in U_k}sk \sum_{j=1}^n(X_v)_{uj} \right) \\
&\overset{\text{\eqref{eq:observe}}}{=}0+\sum_{k=1}^{\ell-1}\sum_{v \in U_k}sk \cdot 1 =s\sum_{k=0}^{\ell-1}ku_k=0, 
\end{align*}
and if $u \in U_h$ with $h \neq 0$, then 
\begin{align*}
\sum_{j=1}^n(M')_{uj}&\overset{\text{\eqref{eq:first}}}{=}\sum_{j=1}^nM_{uj}+\sum_{\substack{1 \leq k \leq \ell-1 \\ k \neq h}}\left( \sum_{v \in U_k}sk \sum_{j=1}^n(X_v)_{uj} \right) 
+ \sum_{v \in U_h}sh \sum_{j=1}^n(X_v)_{uj}
\\
&\overset{\text{\eqref{eq:observe}}}{=}h+\sum_{\substack{1 \leq k \leq \ell-1 \\ k \neq h}}\sum_{v \in U_k} sk \cdot 1+\left(sh(1-n)+ \sum_{v \in U_h \setminus \{u\}}sh \cdot 1\right) \\
&=h+s\sum_{k=0}^{\ell-1}ku_k +sh(1-n)-sh = h - snh = 0. 
\end{align*}

\noindent
{\bf Uniqueness}: 
Assume that there exist modular Eulerian matrices $M,N \in \Alt_n(\ZZ/\ell\ZZ)$ which are switching equivalent. 
Then there exist $a_1,\ldots,a_n \in \ZZ_{>0}$ such that $N \cong M+\sum_{v =1}^n a_vX_v$ in $\Alt_n(\ZZ/\ell\ZZ)$. 
This implies that there exists a permutation $\sigma \in \mathfrak{S}_n$ such that $N'=M+\sum_{v =1}^n a_vX_v$, 
where $N'=(n_{ij}') \in \Alt_n(\ZZ/\ell\ZZ)$ is defined by $n_{ij}'=n_{\sigma(i)\sigma(j)}$. Notice that since $N$ is modular Eulerian, so is $N'$.
Hence, $\sum_{v =1}^n a_vX_v=N'-M$ must be also a modular Eulerian matrix. 
It follows from \eqref{eq:observe} that $\sum_{\substack{1 \leq v \leq n \\ v \neq w}}a_v+a_w(1-n)=0$ holds in $\ZZ/\ell\ZZ$. Namely, 
\[(J-nI)\begin{pmatrix} a_1 \\ \vdots \\ a_n \end{pmatrix}=\begin{pmatrix} 0 \\ \vdots \\ 0 \end{pmatrix}\] 
holds in $(\ZZ/\ell\ZZ)^n$, where $J$ is the all-one matrix and $I$ is the identity matrix. 
Applying Gaussian elimination to $J-nI$ and using the existence of $n^{-1}$ in $\ZZ/\ell\ZZ$, we obtain that 
$\begin{pmatrix} a_1 &\cdots &a_n\end{pmatrix}^T=\begin{pmatrix}a &\cdots &a\end{pmatrix}^T$ for some $a \in \ZZ/\ell\ZZ$.
Since we have $\sum_{v=1}^nX_v=O$, we conclude that $N \cong M+a(\sum_{v =1}^n X_v)=M$, as required. 
\end{proof}

\begin{ex}
Let $n=7$ and $\ell=4$. Consider the matrix 
\[M=
{\small
\begin{pmatrix}
0 &1 &1 &1 &2 &2 &3 \\ 
3 &0 &1 &1 &1 &2 &2 \\ 
3 &3 &0 &1 &1 &1 &2 \\ 
3 &3 &3 &0 &1 &1 &1 \\ 
2 &3 &3 &3 &0 &1 &1 \\ 
2 &2 &3 &3 &3 &0 &1 \\ 
1 &2 &2 &3 &3 &3 &0 
\end{pmatrix}}
\in \Alt_7(\ZZ/4\ZZ). \]
Then we know how to get the corresponding unique modular Eulerian matrix by the proof of Theorem~\ref{thm.Eulermodl}. 
Working with the same notation as in there, we see that 
\begin{align*}
&U_0=\{4\}, \; U_1=\{5\}, \; U_2=\{1,2,6,7\}, \; U_3=\{3\}, \text{ and }s=3 \;\;\text{(i.e., $3 \cdot 7 = 1$ in $\ZZ/4\ZZ$)}. 
\end{align*}
Hence, $\mu_1^6\mu_2^6\mu_3^9\mu_5^3\mu_6^6\mu_7^6(M)$ should become modular Eulerian. In fact, since 
\[2X_1+2X_2+X_3+3X_5+2X_6+2X_7=
{\small
\begin{pmatrix}
0 &0 &3 &2 &1 &0 &0 \\ 
0 &0 &3 &2 &1 &0 &0 \\ 
1 &1 &0 &3 &2 &1 &1 \\ 
2 &2 &1 &0 &3 &2 &2 \\ 
3 &3 &2 &1 &0 &3 &3 \\ 
0 &0 &3 &2 &1 &0 &0 \\ 
0 &0 &3 &2 &1 &0 &0 
\end{pmatrix}},
\]
where $4X_v=O$ for any $v$, %and we have $(\sum_{v=1}^7a_vX_v)_{ij}=-a_i+a_j$, 
we obtain that 
\begin{align*}
\mu_1^6\mu_2^6\mu_3^9\mu_5^3\mu_6^6\mu_7^6(M)=M+2X_1+2X_2+X_3+3X_5+2X_6+2X_7 
=
{\small
\begin{pmatrix}
0 &1 &0 &3 &3 &2 &3 \\ 
3 &0 &0 &3 &2 &2 &2 \\ 
0 &0 &0 &0 &3 &2 &3 \\ 
1 &1 &0 &0 &0 &3 &3 \\ 
1 &2 &1 &0 &0 &0 &0 \\ 
2 &2 &2 &1 &0 &0 &1 \\ 
1 &2 &1 &1 &0 &3 &0 
\end{pmatrix}}.
\end{align*}
\end{ex}

\begin{ex}
Theorem~\ref{thm.Eulermodl} is not true if $n$ is not coprime to $\ell$. 
For example, in the case $\ell=2$ and $n=4$, we have
\[
\xymatrix@R=0.5pc@C=2.0pc{
1 &2& &&1 \ar@{-}[dd]\ar@{-}[rdd]&2\ar@{-}[dd]\ar@{-}[ldd]\\
&&\ar@{~>}[r]_-{\txt{\small $\mu_1\mu_2$}}&&\\
3 &4& &&3&\,4.\\
}
\]
In the case $\ell=3$ and $n=6$, we have  
\[
\xymatrix@R=0.5pc@C=2.0pc{
1 &2&3 &&&1 \ar[dd]\ar[rdd]\ar[rrdd]&2\ar[dd]\ar[ldd]\ar[rdd]&3\ar[lldd]\ar[ldd]\ar[dd]\\
&&&\ar@{~>}[r]_-{\txt{\small $\mu_1\mu_2\mu_3$}}&&&\\
4&5&6& &&4&5&\,6.\\
}
\]
In the case $\ell=4$ and $n=6$, we have  
\[
\xymatrix@R=0.8pc@C=2.4pc{
{\small
\begin{pmatrix}
0 &0 &0 &0 &0 &0 \\
0 &0 &0 &0 &0 &0 \\
0 &0 &0 &0 &0 &0 \\
0 &0 &0 &0 &0 &0 \\
0 &0 &0 &0 &0 &0 \\
0 &0 &0 &0 &0 &0 
\end{pmatrix}}
&\ar@{~>}[r]_-{\txt{\small $\mu_1^2\mu_2^2$}}&&
{\small
\begin{pmatrix}
0 &0 &2 &2 &2 &2 \\
0 &0 &2 &2 &2 &2 \\
2 &2 &0 &0 &0 &0 \\
2 &2 &0 &0 &0 &0 \\
2 &2 &0 &0 &0 &0 \\
2 &2 &0 &0 &0 &0 
\end{pmatrix}}.
}
\]
\end{ex}

On the other hand, if $\ell$ is a prime number, then even if $n$ and $\ell$ are not coprime, we can obtain the following  enumeration result.

\begin{thm}\label{thm.Eulermodlc}
Assume that $\ell$ is a prime number.
Then the number $s_{\ell, n}$ of switching equivalence classes
in $\Alt_n(\ZZ/\ell \ZZ)$ is equal to the number $t_{\ell, n}$ of isomorphism classes of modular Eulerian matrices in $\Alt_n(\ZZ/\ell \ZZ)$.
\end{thm}

\begin{proof}
Since $\ell$ is prime, $K:=\ZZ/\ell\ZZ$ is a finite field.
Let $V_0$ be a $K$-vector space with basis $\{v_i \mid 1\leq i\leq n\}$ and let $V_1$ be a $K$-vector space with basis $\{e_{ij} \mid 1\leq i<j\leq n\}$. 
We denote by $V_0^*$ (resp.\ $V_1^*$) the dual vector space of $V_0$ (resp.\ $V_1$).
Clearly, $f: \Alt_n(\ZZ/\ell \ZZ) \to V_1^*; \; M=(m_{ij}) \mapsto \sum_{1\leq i<j\leq n}m_{ij}e^{*}_{ij}$ is a bijection. 

Let $\psi: V_0^* \to V_1^*$ be the $K$-linear map defined by $\psi(v_i^*)=\sum_{1\leq h<i}e_{hi}^{*}-\sum_{i<j\leq n}e_{ij}^{*}$. 
Then $\psi(v_i^*)=f(X_i)$ for each $i \in [n]$, where $X_i$ is in the sense of Definition~\ref{dfn:switchmodl}(1).
Thus, for $M, M' \in \Alt_n(\ZZ/\ell \ZZ)$,  we see that $M' = \mu_{1}^{i_1}\cdots\mu_{n}^{i_n}(M)$
for some $i_1,\dots, i_n \in \ZZ_{>0}$ if and only if 
$f(M)-f(M') \in \Im \psi$.

We define the action of ${\mathfrak S}_n$ on $V_1^*/\Im \psi$ by 
$\sigma \left(\overline{\sum_{1\leq i<j\leq n}m_{ij}e^{*}_{ij}}\right)=\overline{\sum_{1\leq i<j\leq n}m_{ij}e^{*}_{\sigma(i)\sigma(j)}}$ for $\sigma \in {\mathfrak S}_n$, where we consider $e^{*}_{\sigma(i)\sigma(j)}=-e^{*}_{\sigma(j)\sigma(i)}$ if $\sigma(i)>\sigma(j)$.
Then it follows from Burnside's lemma that
\begin{align}\label{eq.b1}
s_{\ell, n}=|(V_1^*/\Im \psi)/{\mathfrak S}_n|=\frac{1}{n!}\sum_{\sigma \in {\mathfrak S}_n}|(V_1^*/\Im \psi)^{\sigma}|,
\end{align}
where $(V_1^*/\Im \psi)^{\sigma}$ is the set of elements in $V_1^*/\Im \psi$ that are fixed by $\sigma$. 

Let $\phi: V_1 \to V_0$ be the $K$-linear map defined by $\phi(e_{ij})=v_j-v_i$. Then a bijection 
$g:\Alt_n(\ZZ/\ell \ZZ) \to V_1;\ M=(m_{ij}) \mapsto \sum_{1\leq i<j\leq n}m_{ij}e_{ij}$ 
implies a bijection
\begin{align*}
\{\text{modular Eulerian matrices in $\Alt_n(\ZZ/\ell \ZZ)$}\} \to \Ker \phi
\end{align*}
by restricting $g$ to modular Eulerian matrices. 

We define the action of ${\mathfrak S}_n$ on $\Ker \phi$ by
$\sigma (\sum_{1\leq i<j\leq n}m_{ij}e_{ij})=\sum_{1\leq i<j\leq n}m_{ij}e_{\sigma(i)\sigma(j)}$ for $\sigma \in {\mathfrak S}_n$, where we consider $e_{\sigma(i)\sigma(j)}=-e_{\sigma(j)\sigma(i)}$ if $\sigma(i)>\sigma(j)$. Then it follows from Burnside's lemma that
\begin{align}\label{eq.b2}
t_{\ell, n}=|\Ker \phi/{\mathfrak S}_n|=\frac{1}{n!}\sum_{\sigma \in {\mathfrak S}_n}|(\Ker \phi)^{\sigma}|,
\end{align}
where  $(\Ker \phi)^{\sigma}$ is the set of elements in $\Ker \phi$ that are fixed by $\sigma$. 

By \eqref{eq.b1} and \eqref{eq.b2}, it is enough to show that $|(V_1^*/\Im \psi)^{\sigma}|=|(\Ker \phi)^{\sigma}|$.
Since the diagram
\[
\xymatrix@R=2.5pc@C=4pc{
V_1 \ar[d]^{\text{can.}}_{\cong} \ar[r]^-{\phi} &V_0 \ar[d]^{\text{can.}}_{\cong} \\
V_1^{**} \ar[r]^-{\psi^*} &V_0^{**}
}
\]
commutes, there exists an induced isomorphism $g_1: \Ker \phi \to  (V_1^*/\Im \psi)^*$. 
Let $g_2$ be an isomorphism $V_1^*/\Im \psi \to(V_1^*/\Im \psi)^*$ defined by mapping a basis to its dual basis. Then we have the composition of isomorphisms
\[\xymatrix@R=2pc@C=2pc{
g: \Ker \phi  \ar[r]^-{g_1}  &(V_1^*/\Im \psi)^*  \ar[r]^-{g_2^{-1}} &V_1^*/\Im \psi
}.
\]
One can check that $w=\sigma(w)$ if and only if $g(w)=\sigma(g(w))$
for any $w \in \Ker \phi$, so we get $|(\Ker \phi)^{\sigma}|=|(V_1^*/\Im \psi)^{\sigma}|$.
This concludes the proof.
\end{proof}

Combining with Theorem \ref{thm.Eulermodl}, we obtain the following result.

\begin{cor}\label{cor.Eulermodlc}
If $\ell$ and $n$ are coprime or $\ell$ is prime, then 
the number $s_{\ell,n}$ of switching equivalence classes
in $\Alt_n(\ZZ/\ell \ZZ)$ is equal to the number $t_{\ell,n}$ of isomorphism classes of modular Eulerian matrices in $\Alt_n(\ZZ/\ell \ZZ)$.
\end{cor}

\begin{rem}\label{rem:num}
(1) The $\ell=2$ case of Theorem \ref{thm.Eulermodlc}, that is,
the coincidence of 
the number $s_{2,n}$ of switching equivalence classes of graphs and the number $t_{2,n}$ of unlabeled Eulerian graphs, was proved by Seidel \cite{S} for $n$ odd and by
Mallows and Sloane \cite{MS} for all $n$.
In this case, the values of $s_{2,n}=t_{2,n}$ for $1\leq n\leq 11$ are known as
\[
\renewcommand{\arraystretch}{1.2}
\begin{array} {c|cccccccccccccc}
n &1&2&3&4&5&6&7&8&9&10&11 \\ \hline
s_{2,n}=t_{2,n}&1&1&2&3&7&16&54&243&2038&33120&1182004
\end{array}
\]
(see \cite[Table 1]{MS} and  A002854 of OEIS \cite{OE}).

(2) The $\ell=3$ case of Theorem \ref{thm.Eulermodlc}, that is,
the coincidence of 
the number $s_{3,n}$ of switching equivalence classes of digraphs and the number $t_{3,n}$ of unlabeled  modular Eulerian digraphs, was proved by 
Cheng and Wells Jr.\ \cite{CW}.
In this case, the values of $s_{3,n}=t_{3,n}$ for $1\leq n\leq 11$ are known as
\[
\renewcommand{\arraystretch}{1.2}
\begin{array} {c|cccccccccccccc}
n &1&2&3&4&5&6&7&8&9&10&11 \\ \hline
s_{3,n}=t_{3,n}&1&1&2&4&14&120&3222&271287 &64154817 &41653775052&74220906305025 
\end{array}
\]
(see \cite[Table 1]{CW} and  A240973 of OEIS \cite{OE}).
\end{rem}

\begin{ex}
We consider the case where $\ell=4$.
Unlike in Remark \ref{rem:num}, we do not know whether $s_{4,n}=t_{4,n}$ holds or not in general, 
but we know this equality if $n \equiv 1 \text { or }3 \pmod{4}$ by Corollary~\ref{cor.Eulermodlc}.  
The number $t_{4,n}$ of modular Eulerian matrices %in $\Alt_n(\ZZ/4\ZZ)$
for $1\leq n\leq 6$ can be computed as follows: 
\[
\renewcommand{\arraystretch}{1.2}
\begin{array} {c|cccccc}
n &1&2&3&4&5 &6 \\ \hline
t_{4,n}&1&1&3&8&62 &1760
\end{array}
\]
Corollary~\ref{cor.Eulermodlc} guarantees
$s_{4,1}=t_{4,1}$, $s_{4,3}=t_{4,3}$, and $s_{4,5}=t_{4,5}$.
Moreover, we have checked $s_{4,2}=t_{4,2}$ and $s_{4,4}=t_{4,4}$ by hand. 
We expect $s_{4,n}=t_{4,n}$ for any $n$, though we do not have a proof.
\end{ex}

The discussion in this section yields the following results on skew polynomial algebras.

\begin{thm}
\textnormal{(1)} Let $S_\omega$ be a standard graded skew polynomial algebra in $n$ variables at $\ell$-th roots of unity. If $\ell$ and $n$ are coprime, then there exists a standard graded skew polynomial algebra $S_{\omega'}$ in $n$ variables at $\ell$-th roots of unity such that
\begin{enumerate}
\item[(a)] $\GrMod S_\omega$ is equivalent to $\GrMod S_{\omega'}$ and
\item[(b)] the E-matrix $M_{\omega'}$ is a modular Eulerian matrix in $\Alt_n(\ZZ/\ell \ZZ)$,
\end{enumerate}
which is unique up to isomorphism of graded algebras.

\textnormal{(2)}
If $\ell$ and $n$ are coprime or $\ell$ is prime, then
the number of standard graded skew polynomial algebras in $n$ variables at $\ell$-th roots of unity up to equivalence of graded module categories is equal to the number of isomorphism classes of modular Eulerian matrices in $\Alt_n(\ZZ/\ell \ZZ)$. 
\end{thm}

\begin{proof}
(1) Let $M_{\omega}$ be the E-matrix of $S_{\omega}$.
By Theorem \ref{thm.Eulermodl}, there exists a modular Eulerian matrix $N$ in the switching equivalence class of $M_\omega$, which is unique up to isomorphism. Let $S_{\omega'}$ be a standard graded skew polynomial algebra 
defined by $M_{\omega'}=N$. By Proposition \ref{prop.skewisol} and Theorem \ref{thm.grmodl}, $S_{\omega'}$ is as desired.

(2) This is a direct consequence of Theorem \ref{thm.grmodl} and Corollary \ref{cor.Eulermodlc}.
\end{proof}

\section{Point simplicial complexes of skew polynomial algebras at cube roots of unity}

In this section, we study the point simplicial complexes $\Delta_{\omega}$ of standard graded skew polynomial algebras $S_{\omega}$ at cube roots of unity by using digraphs (recall that this is connected to studying the point varieties $\Gamma_{\omega}$). 

As mentioned in Example \ref{ex.3graph}, a skew-symmetric matrix $M \in \Alt_n(\ZZ/3 \ZZ)$ can be regarded as the adjacency matrix of the digraph $G(M)$, so we frequently
use the digraph $G(M_\omega)$ to study the E-matrix $M_\omega \in \Alt_n(\ZZ/3\ZZ)$ of $S_{\omega}$.
Furthermore, for simplicity of notation, we denote a face $\{i_1,i_2,\dots,i_m\}$ of $\Delta_{\omega}$ as $i_1i_2\cdots i_m$.

To examine point simplicial complexes, the following operation on a digraph is useful.

\begin{dfn}
For a given digraph $G$ and its vertex $v$, there exists a unique digraph $\widetilde{G}$ which is given by iterated switching applied to $G$ 
such that $v$ is an isolated vertex in $\widetilde{G}$. (See Proposition~\ref{prop:iso}.) 
We call $\widetilde{G}$ the \emph{isolation} of $G$ at $v$ and write it as $I_v(G)$.
\end{dfn}
For example, let $G$ be the following: 
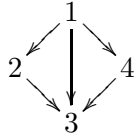
\[\xy /r1.75pc/: 
{\xypolygon4{~={90}~*{\xypolynode}~>{}}},
\ar@{->}"1";"2",
\ar@{->}"1";"3",
\ar@{->}"1";"4",
\ar@{->}"2";"3",
\ar@{->}"4";"3",
\endxy\]
Then each of the isolations is as follows: 
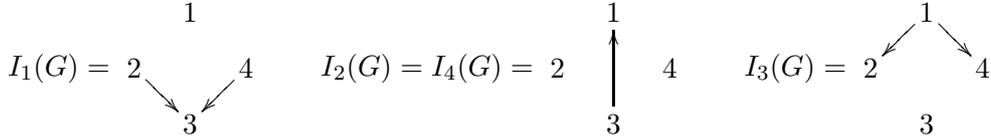
\[I_1(G)=
\xy /r1.75pc/: 
{\xypolygon4{~={90}~*{\xypolynode}~>{}}},
\ar@{->}"2";"3",
\ar@{->}"4";"3",
\endxy\quad\quad
I_2(G)=I_4(G)=
\xy /r1.75pc/: 
{\xypolygon4{~={90}~*{\xypolynode}~>{}}},
\ar@{->}"3";"1",
\endxy\quad\quad
I_3(G)=\xy /r1.75pc/: 
{\xypolygon4{~={90}~*{\xypolynode}~>{}}},
\ar@{->}"1";"2",
\ar@{->}"1";"4",
\endxy
\]

First we show the following result; compare with
Theorem \ref{thm:pm1} (which is for $\ell=2$) and Example \ref{ex:swsc} (which is for $\ell\geq4$).

\begin{thm}\label{thm.classps}
Let $S_\omega$ and $S_{\omega'}$ be standard graded skew polynomial algebras in $n$ variables at cube roots of unity. If $n \leq 5$, then the following are equivalent.
\begin{enumerate}
\item[(\rnum{1})] $\GrMod S_{\omega}$ is equivalent to $\GrMod S_{\omega'}$.
\item[(\rnum{2})] $\Delta_{\omega}$ is isomorphic to  $\Delta_{\omega'}$.
\end{enumerate}
\end{thm}

\begin{proof}
By Propositions \ref{prop.grps} and \ref{prop.pssc}, we have only to show $\textnormal{(\rnum{2})} \Rightarrow \textnormal{(\rnum{1})}$. By Theorem \ref{thm.grmodl}, it is enough to verify that
\begin{align}\tag{$\star$} \label{eq.s}
\text{if $M_\omega$ is not switching equivalent to $M_{\omega'}$, then $\Delta_{\omega}$ is not isomorphic to  $\Delta_{\omega'}$.}
\end{align}

If $n$ is $1$ or $2$, then there exists only one switching equivalence class of digraphs on with $n$ vertices, so \eqref{eq.s} is trivially true.

We consider the case $n=3$. Then $M_{\omega}$ is switching equivalent to exactly one of the adjacency matrices of the following digraphs: 
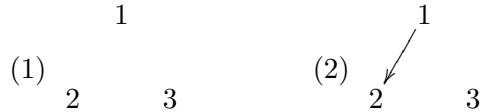
\begin{center}
(1) $\xy /r1.75pc/: 
{\xypolygon3{~={90}~*{\xypolynode}~>{}}},
\endxy$
\qquad\qquad
(2) $\xy /r1.75pc/: 
{\xypolygon3{~={90}~*{\xypolynode}~>{}}},
\ar@{->}"1";"2",
\endxy$
\end{center}
In each of the above cases, the point simplicial complex $\Delta_{\omega}$ is given as follows: 
\begin{enumerate}
\item $\calF(\Delta_{\omega})=\{123\}$ 
\item $\calF(\Delta_{\omega})=\{12,13,23\}$. 
\end{enumerate}
These two simplicial complexes are not isomorphic, so \eqref{eq.s} is true.

We consider the case $n=4$.
There are four modular Eulerian digraphs with $4$ vertices up to isomorphism (see Remark \ref{rem:num}(2)).
Namely, they are given as follows; so $M_{\omega}$ is switching equivalent to exactly one of the adjacency matrices of the following digraphs by Theorem \ref{thm.Eulermodl}:

\begin{center}
(1) $\xy /r1.75pc/: 
{\xypolygon4{~={90}~*{\xypolynode}~>{}}},
\endxy$
\qquad\quad
(2) $\xy /r1.75pc/: 
{\xypolygon4{~={90}~*{\xypolynode}~>{}}},
\ar@{->}"1";"2",
\ar@{->}"2";"3",
\ar@{->}"3";"1",
\endxy$
\qquad\quad
(3) $\xy /r1.75pc/: 
{\xypolygon4{~={90}~*{\xypolynode}~>{}}},
\ar@{->}"1";"2",
\ar@{->}"2";"3",
\ar@{->}"3";"4",
\ar@{->}"4";"1",
\endxy$
\qquad\quad
(4) $
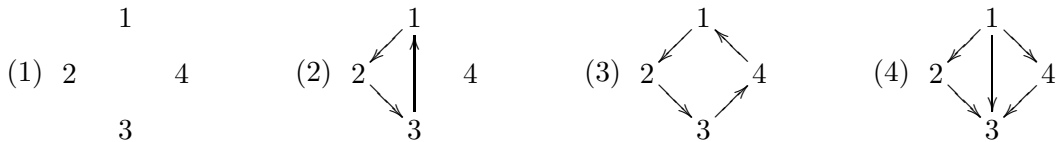
\xy /r1.75pc/: 
{\xypolygon4{~={90}~*{\xypolynode}~>{}}},
\ar@{->}"1";"2",
\ar@{->}"1";"3",
\ar@{->}"1";"4",
\ar@{->}"2";"3",
\ar@{->}"4";"3",
\endxy$
\end{center}
In each of the above cases, the point simplicial complex $\Delta_{\omega}$ is given as follows\footnote{In \cite[Section 4.2]{BDL} (resp.\,\cite[Section 4.3]{BDL}),  the classification of the point varieties (in other words, the point simplicial complexes) of skew polynomial algebras in $4$ (resp.\,$5$) variables was given.
The labels used here are the same as those used in that classification. }:
\begin{enumerate}
\item[(1)] $\calF(\Delta_{\omega})=\{1234\}$ \quad (label: $0$)
\item[(2)] $\calF(\Delta_{\omega})=\{123, 14, 24, 34\}$ \quad (label: $2$)
\item[(3)] $\calF(\Delta_{\omega})=\{12, 13, 14, 23, 24, 34\}$ \quad (label: $3$).  
\item[(4)] $\calF(\Delta_{\omega})=\{124, 234, 13\}$ \quad (label: $1$)
\end{enumerate}
These four simplicial complexes are not isomorphic to each other, so \eqref{eq.s} is true.

We consider the case $n=5$. There are fourteen modular Eulerian digraphs with $5$ vertices up to isomorphism (see Remark \ref{rem:num}(2)).
Namely, they are given as follows; so $M_{\omega}$ is switching equivalent to exactly one of the adjacency matrices of the following digraphs by Theorem \ref{thm.Eulermodl}:

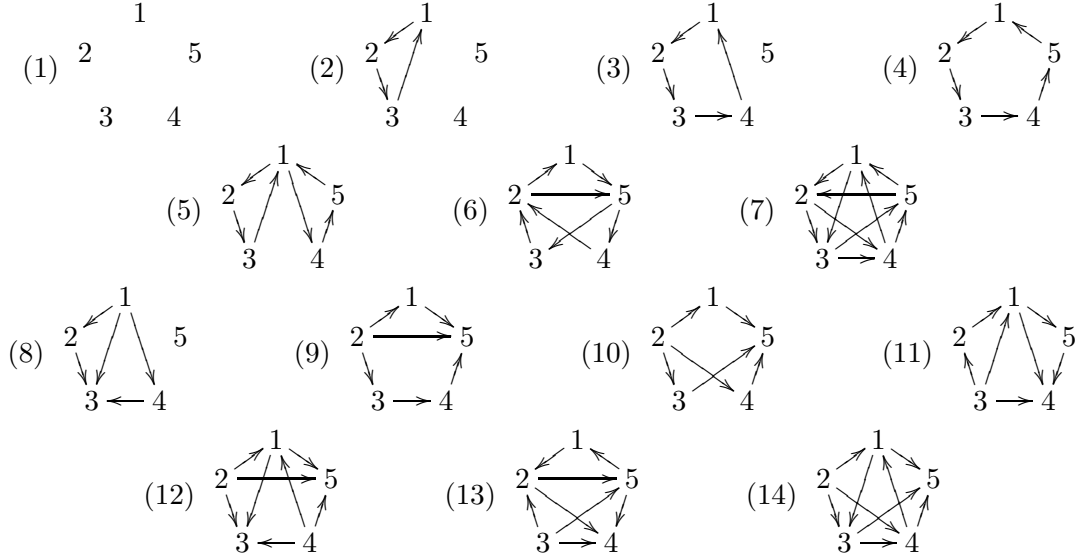
\begin{center}
(1) $\xy /r1.8pc/: 
{\xypolygon5{~={90}~*{\xypolynode}~>{}}},
\endxy$
\qquad\quad
(2) $\xy /r1.8pc/: 
{\xypolygon5{~={90}~*{\xypolynode}~>{}}},
\ar@{->}"1";"2",
\ar@{->}"2";"3",
\ar@{->}"3";"1",
\endxy$
\qquad\quad
(3) $\xy /r1.8pc/: 
{\xypolygon5{~={90}~*{\xypolynode}~>{}}},
\ar@{->}"1";"2",
\ar@{->}"2";"3",
\ar@{->}"3";"4",
\ar@{->}"4";"1",
\endxy$
\qquad\quad
(4) $\xy /r1.8pc/: 
{\xypolygon5{~={90}~*{\xypolynode}~>{}}},
\ar@{->}"1";"2",
\ar@{->}"2";"3",
\ar@{->}"3";"4",
\ar@{->}"4";"5",
\ar@{->}"5";"1",
\endxy$
\\
(5) $\xy /r1.8pc/: 
{\xypolygon5{~={90}~*{\xypolynode}~>{}}},
\ar@{->}"1";"2",
\ar@{->}"2";"3",
\ar@{->}"3";"1",
\ar@{->}"1";"4",
\ar@{->}"4";"5",
\ar@{->}"5";"1",
\endxy$
\qquad\quad
(6) $\xy /r1.8pc/: 
{\xypolygon5{~={90}~*{\xypolynode}~>{}}},
\ar@{->}"2";"1",
\ar@{->}"2";"5",
\ar@{->}"1";"5",
\ar@{->}"3";"2",
\ar@{->}"5";"3",
\ar@{->}"4";"2",
\ar@{->}"5";"4",
\endxy$
\qquad\quad
(7) $\xy /r1.8pc/: 
{\xypolygon5{~={90}~*{\xypolynode}~>{}}},
\ar@{->}"1";"2",
\ar@{->}"2";"3",
\ar@{->}"3";"4",
\ar@{->}"4";"5",
\ar@{->}"5";"1",
\ar@{->}"1";"3",
\ar@{->}"3";"5",
\ar@{->}"5";"2",
\ar@{->}"2";"4",
\ar@{->}"4";"1",
\endxy$
\\
(8) $\xy /r1.8pc/: 
{\xypolygon5{~={90}~*{\xypolynode}~>{}}},
\ar@{->}"1";"2",
\ar@{->}"2";"3",
\ar@{->}"1";"3",
\ar@{->}"1";"4",
\ar@{->}"4";"3",
\endxy$
\qquad\quad
(9) $\xy /r1.8pc/: 
{\xypolygon5{~={90}~*{\xypolynode}~>{}}},
\ar@{->}"2";"1",
\ar@{->}"1";"5",
\ar@{->}"2";"5",
\ar@{->}"2";"3",
\ar@{->}"3";"4",
\ar@{->}"4";"5",
\endxy$
\qquad\quad
(10) $\xy /r1.8pc/: 
{\xypolygon5{~={90}~*{\xypolynode}~>{}}},
\ar@{->}"2";"1",
\ar@{->}"2";"3",
\ar@{->}"2";"4",
\ar@{->}"1";"5",
\ar@{->}"3";"5",
\ar@{->}"4";"5",
\endxy$
\qquad\quad
(11) $\xy /r1.8pc/: 
{\xypolygon5{~={90}~*{\xypolynode}~>{}}},
\ar@{->}"2";"1",
\ar@{->}"3";"2",
\ar@{->}"3";"1",
\ar@{->}"1";"5",
\ar@{->}"5";"4",
\ar@{->}"1";"4",
\ar@{->}"3";"4",
\endxy$
\\
(12) $\xy /r1.8pc/: 
{\xypolygon5{~={90}~*{\xypolynode}~>{}}},
\ar@{->}"1";"3",
\ar@{->}"1";"5",
\ar@{->}"2";"1",
\ar@{->}"2";"5",
\ar@{->}"2";"3",
\ar@{->}"4";"5",
\ar@{->}"4";"1",
\ar@{->}"4";"3",
\endxy$
\qquad\quad
(13) $\xy /r1.8pc/: 
{\xypolygon5{~={90}~*{\xypolynode}~>{}}},
\ar@{->}"1";"2",
\ar@{->}"2";"4",
\ar@{->}"2";"5",
\ar@{->}"3";"2",
\ar@{->}"3";"4",
\ar@{->}"3";"5",
\ar@{->}"5";"1",
\ar@{->}"5";"4",
\endxy$
\qquad\quad
(14) $\xy /r1.8pc/: 
{\xypolygon5{~={90}~*{\xypolynode}~>{}}},
\ar@{->}"1";"3",
\ar@{->}"1";"5",
\ar@{->}"2";"1",
\ar@{->}"2";"3",
\ar@{->}"2";"4",
\ar@{->}"3";"4",
\ar@{->}"3";"5",
\ar@{->}"4";"5",
\ar@{->}"4";"1",
\endxy$
\end{center}
In each of the above cases, the point simplicial complex $\Delta_{\omega}$ is given as follows\footnotemark[1]:
\begin{enumerate}
\item[(1)] $\calF(\Delta_{\omega})=\{12345\}$
\quad (label: $0$) 
\item[(2)] $\calF(\Delta_{\omega})=\{123, 145, 245, 345\}$ 
\quad (label: $2_a$)
\item[(3)] 
$\calF(\Delta_{\omega})=\{135, 245, 12, 14, 23, 34\}$ 
\quad (label: $4_b$)
\item[(4)] 
$\calF(\Delta_{\omega})=\{12, 13, 14, 15, 23, 24, 25, 34, 35, 45\}$
\quad (label: $6$)
\item[(5)] 
$\calF(\Delta_{\omega})=\{123, 124, 135, 145, 25, 34\}$
\quad (label: $2_b$)
\item[(6)] 
$\calF(\Delta_{\omega})=\{2345, 134, 12, 15\}$ 
\quad (label: $2_c$)
\item[(7)] 
$\calF(\Delta_{\omega})=\{124, 134, 135, 235, 245\}$
\quad (label: $1_b$)
\item[(8)] 
$\calF(\Delta_{\omega})=\{124, 234, 245, 13, 15, 35\}$
\quad (label: $3_c$)
\item[(9)] 
$\calF(\Delta_{\omega})=\{123, 145, 235, 245, 34\}$
\quad (label: $2_d$) 
\item[(10)] 
$\calF(\Delta_{\omega})=\{1234, 1345, 25\}$
\quad (label: $1_c$)
\item[(11)] 
$\calF(\Delta_{\omega})=\{234, 345, 12, 13, 14, 15, 25\}$
\quad (label: $4_a$)
\item[(12)] 
$\calF(\Delta_{\omega})=\{2345, 124, 135\}$ 
\quad (label: $1_a$)
\item[(13)] 
$\calF(\Delta_{\omega})=\{123, 125, 145, 24, 34, 35\}$
\quad (label: $3_a$)
\item[(14)]
$\calF(\Delta_{\omega})=\{134, 12, 15, 23, 24, 25, 35, 45\}$ 
\quad (label: $5$)
\end{enumerate}
These fourteen simplicial complexes are not isomorphic to each other, so \eqref{eq.s} is true.

Hence the proof is complete. 
\end{proof}

\begin{ex}\label{ex.simpcomp}
If we remove the assumption that $n\leq 5$, there exist counterexamples to the statement of Theorem \ref{thm.classps}.

(1) Let us consider the following case when $n=6$:
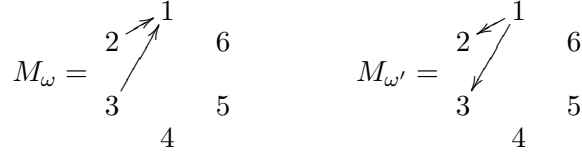
\begin{align*}
G(M_{\omega})=\xy /r2pc/: 
{\xypolygon6{~={90}~*{\xypolynode}~>{}}},
\ar@{->}"2";"1",
\ar@{->}"3";"1",
\endxy
\qquad\qquad
G(M_{\omega'})=\xy /r2pc/: 
{\xypolygon6{~={90}~*{\xypolynode}~>{}}},
\ar@{->}"1";"2",
\ar@{->}"1";"3",
\endxy
\end{align*}
Then we can straightforwardly check that $\calF(\Delta_{\omega}) = \calF(\Delta_{\omega'})=\{23456, 1456, 123\}$. 
In particular, $\Delta_\omega=\Delta_{\omega'}$. 

We prove that $M_{\omega}$ and $M_{\omega'}$ are not switching equivalent. 
Suppose that $M_{\omega}$ is switching equivalent to $M_{\omega'}$.
Since $G(M_{\omega'})$ has an isolated vertex, it must be isomorphic to an isolation $I_v(G(M_{\omega}))$ for some $v \in V(G(M_{\omega}))$. 
Now all isolations of $G(M_\omega)$ are as follows: 
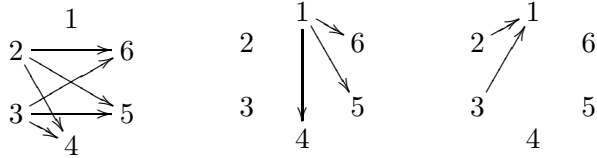
\begin{align*}
\xy /r2pc/: 
{\xypolygon6{~={90}~*{\xypolynode}~>{}}},
\ar@{->}"2";"4",
\ar@{->}"2";"5",
\ar@{->}"2";"6",
\ar@{->}"3";"4",
\ar@{->}"3";"5",
\ar@{->}"3";"6",
\endxy
\qquad\quad
\xy /r2pc/: 
{\xypolygon6{~={90}~*{\xypolynode}~>{}}},
\ar@{->}"1";"6",
\ar@{->}"1";"4",
\ar@{->}"1";"5",
\endxy
\qquad\quad
\xy /r2pc/: 
{\xypolygon6{~={90}~*{\xypolynode}~>{}}},
\ar@{->}"2";"1",
\ar@{->}"3";"1",
\endxy
\end{align*}
Clearly, none of those is isomorphic to $G(M_{\omega'})$. This is a contradiction. 
Therefore, we conclude that $M_{\omega}$ is not switching equivalent to $M_{\omega'}$, and thus $\GrMod S_{\omega}$ is not equivalent to $\GrMod S_{\omega'}$.

(2)
Let us consider the following case when $n=7$: 
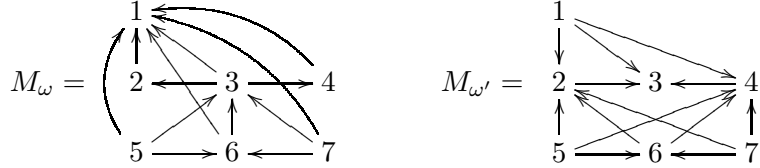
\[
\xymatrix@R=1.0pc@C=2.0pc{
&1\\
G(M_{\omega})=\hspace*{-5mm}&2\ar[u]&3\ar[ul]\ar[l]\ar[r]&4\ar@/_12pt/[ull]\\
&5\ar[r]\ar[ur]\ar@/^12pt/[uu]&6\ar[luu]\ar[u]&7\ar[l]\ar[lu] \ar@/_12pt/[lluu]
}
\qquad\quad 
\xymatrix@R=1.0pc@C=2.0pc{
&1\ar[d]\ar[rd]\ar[rrd]\\
G(M_{\omega'})=\hspace*{-7mm}&2\ar[r]&3&4\ar[l]\\
&5\ar[u]\ar[rru]\ar[r]&6\ar[lu]\ar[ru]&7\ar[u]\ar[l]\ar[ull]
}
\]
Then one can check that $\calF(\Delta_{\omega}) = \calF(\Delta_{\omega'})=\{
12457,
1246,
234, 357, 567,
13, 36
\}$. 
In particular, $\Delta_\omega=\Delta_{\omega'}$. 

Since both $G(M_{\omega})$ and $G(M_{\omega'})$ are non-isomorphic modular Eulerian digraphs, it follows from Theorem \ref{thm.Eulermodl}
that $M_{\omega}$ and $M_{\omega'}$ are not switching equivalent, so $\GrMod S_{\omega}$ and $\GrMod S_{\omega'}$ are not equivalent.
\end{ex}

The point simplicial complex $\Delta_{\omega}$ can be viewed as a collection of information on the underlying graphs of isolations of $M_\omega$.

\begin{prop}\label{prop:isol1}
Let $S_\omega$ and $S_{\omega'}$ be standard graded skew polynomial algebras in $n$ variables at cube roots of unity. Then the following are equivalent.
\begin{enumerate}
\item[(\rnum{1})] $\Delta_{\omega}$ is isomorphic to  $\Delta_{\omega'}$.
\item[(\rnum{2})] There exists a permutation $\sigma \in \mathfrak{S}_n$ such that $\sigma$ induces an isomorphisms between the underlying graphs of $I_v(G(M_\omega))$ and $I_{\sigma(v)}(G(M_{\omega'}))$ for every $v \in [n]$.
\end{enumerate}
\end{prop}

\begin{proof}
(1) $\Rightarrow$ (2): Let $\sigma \in \mathfrak{S}_n$ be a permutation inducing an isomorphism between $\Delta_\omega$ and $\Delta_{\omega'}$. 
Fix an arbitrary $v \in [n]$ and let $v'=\sigma(v)$. Our goal is to show that the underlying graphs of $I_v(G(M_\omega))$ and $I_{v'}(G(M_{\omega'}))$ are isomorphic. 
Let $a,b \in [n]$ with $a \neq b$. 
\begin{itemize}
\item If $v \in \{a,b\}$, since $v$ is an isolated vertex, there is no directed edge between $a$ and $b$. 
Since $v'$ is also an isolated vertex, there is no directed edge between $\sigma(a)$ and $\sigma(b)$, either. 
\item Let $v \not\in \{a,b\}$. We notice that whether $abv$ is a face of $\Delta_\omega$ is equivalent to whether $\sigma(a)\sigma(b)v'$ is a face of $\Delta_{\omega'}$. 
Since $v$ is an isolated vertex, the adjacency of $a$ and $b$ in the underlying graph of $I_v(G(M_\omega))$ 
is determined by whether $abv$ is a face in $\Delta_\omega$ or not. 
This implies that the adjacency of $a$ and $b$ is equivalent to that of $\sigma(a)$ and $\sigma(b)$ in the underlying graph of $I_{v'}(G(M_{\omega'}))$. 
\end{itemize}
These mean that $\sigma$ induces an isomorphism between the underlying graphs of $I_v(G(M_\omega))$ and $I_{v'}(G(M_{\omega'}))$.  

(2) $\Rightarrow$ (1): Let $\sigma$ be a permutation inducing an isomorphism 
between the underlying graphs of $I_v(G(M_\omega))$ and $I_{\sigma(v)}(G(M_{\omega'}))$ for every $v \in [n]$. 

For any distinct $a,b,c \in [n]$, consider $I_c(G(M_\omega))$. 
Since the adjacency of $a$ and $b$ in $I_c(G(M_\omega))$ is equivalent to that of $\sigma(a)$ and $\sigma(b)$ in $I_{\sigma(c)}(G(M_{\omega'}))$, 
we obtain that whether $abc$ is a face of $\Delta_\omega$ is equivalent to whether $\sigma(a)\sigma(b)\sigma(c)$ is a face of $\Delta_{\omega'}$, as required. 
\end{proof}

For a (di)graph $G$, we say that a subset $T$ of the vertex set $V(G)$ is an \emph{independent set} if there is no edge between $u$ and $v$ for any distinct $u,v \in T$. 
The \emph{independence number} $\alpha(G)$ of $G$ is defined by
$\alpha(G)=\max\{|T| \mid T \text{ is an independent set of }G\}$. 
The facets and the dimension of $\Delta_\omega$ (which corresponds to the irreducible components and the dimension of the point variety $\Gamma_\omega$, respectively) 
can be calculated from isolations of $M_\omega$ as follows. 

\begin{prop}\label{prop:isol2}
Let $S_\omega$ be a standard graded skew polynomial algebra in $n$ variables at cube roots of unity. 
Then $\Delta_\omega$ consists of the independent sets of $I_v(G(M_\omega))$ for all $v \in [n]$.
In particular, we have
\[
\dim \Delta_\omega
= \max\{\alpha(I_v(G(M_\omega))) \mid v \in [n]\} - 1. 
\]
\end{prop}

\begin{proof}
Let $T \subset [n]$ be an independent set of $I_v(G(M_\omega))$ for some $v \in [n]$. Then $T$ is a face of $\Delta_\omega$ by definition. 

Let $F \in \Delta_\omega$ and fix $u \in F$. Our goal is to show that $F$ is an independent set of $I_u(G(M_\omega))$. 
Let $S_{\omega'}$ be a standard graded skew polynomial algebra at cube roots of unity with $G(M_{\omega'})=I_u(G(M_\omega))$. 
Then $\Delta_{\omega'}=\Delta_{\omega}$, so $F$ is also a face of $\Delta_{\omega'}$. 
For any distinct $a, b \in F\setminus\{u\}$, we see that $a$ is not adjacent to $b$ in $G(M_{\omega'})$ because $abu \subset F \in \Delta_{\omega'}$ and $u$ is an isolated vertex of $G(M_{\omega'})$. Therefore, $F$ is an independent set in $G(M_{\omega'})=I_u(G(M_\omega))$. 
\end{proof}

\begin{ex}
Let us consider $\Delta_\omega$ for $M_\omega$ given in Example~\ref{ex.simpcomp} (1) and (2), respectively: 

(1) According to the isolations provided there, we see that the maximal independent sets are as follows: 
\[I_1(G(M_\omega)): 123, 1456 \quad I_2(G(M_\omega)): 123, 23456 \quad I_4(G(M_\omega)): 1456, 23456\]
Hence, we have $\calF(\Delta_\omega)=\{23456,1456,123\}$ and $\dim\Delta_\omega=4$. 

(2) All isolations of $G(M_\omega)$ are as follows: 
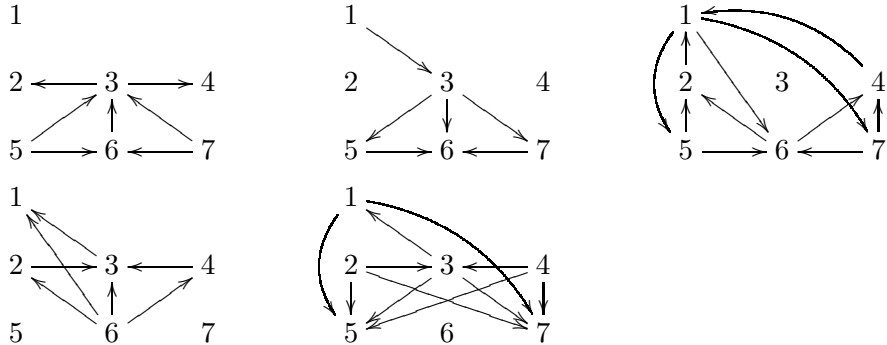
\begin{align*}
&\xymatrix@R=1.0pc@C=2.0pc{
&1\\
&2 &3\ar[l]\ar[r]&4 \\
&5\ar[r]\ar[ur]&6\ar[u]&7\ar[l]\ar[ul]
}\quad
\xymatrix@R=1.0pc@C=2.0pc{
&1\ar[dr]\\
&2 &3\ar[dl]\ar[d]\ar[dr]&4 \\
&5\ar[r]&6&7\ar[l]
}\quad
\xymatrix@R=1.0pc@C=2.0pc{
&1\ar@/_12pt/[dd]\ar[ddr]\ar@/^12pt/[rrdd]\\
&2\ar[u]&3&4\ar@/_12pt/[ull]\\
&5\ar[r]\ar[u]&6\ar[lu]\ar[ru]&7\ar[l]\ar[u] 
}\\
&\xymatrix@R=1.0pc@C=2.0pc{
&1\\
&2\ar[r] &3\ar[lu]&4\ar[l] \\
&5&6\ar[u]\ar[lu]\ar[ru]\ar[luu]&7
}\quad
\xymatrix@R=1.0pc@C=2.0pc{
&1\ar@/_12pt/[dd]\ar@/^12pt/[rrdd]\\
&2\ar[r]\ar[d]\ar[rrd]&3\ar[lu]\ar[ld]\ar[rd]&4\ar[l]\ar[d]\ar[lld] \\
&5 &6 &7
}
\end{align*}
We see that the maximal independent sets are as follows: 
\begin{align*}
&I_1(G(M_\omega)): 12457, 13, 1246 \quad I_2(G(M_\omega)): 12457, 234, 1246 \quad I_3(G(M_\omega)): 13, 234, 357, 36 \\
&I_5(G(M_\omega)): 567, 12457, 357, 567 \quad I_6(G(M_\omega)): 1246, 36, 567
\end{align*}
Hence, we have $\calF(\Delta_\omega)=\{12457, 13, 1246, 234, 357, 36, 567\}$ and $\dim\Delta_\omega=4$. 
\end{ex}

\begin{rem}\label{rem:isomat}
For the purpose of this section, we focused on isolations of a digraph. 
However, it is also possible to introduce ``isolations of skew-symmetric matrices'' as in Proposition~\ref{prop:iso}. 
%In fact, for a given $M \in \Alt_n(\ZZ/\ell\ZZ)$ and $v \in [n]$, there exists a unique $\widetilde{M}$ which is switching equivalent to $M$ 
%such that $v$-th row and column are all $0$. 
We can also regard a subset $T \subset [n]$ as an ``independent set'' if the subsquare matrix restricted to the rows and columns indexed by $T$ is $O$. 
Then we can apply similar proofs for Propositions~\ref{prop:isol1} and \ref{prop:isol2}, so the same statements also hold even in the case $\ell \geq 4$. 
\end{rem}

\begin{rem}\label{rem:isoalg}
In addition to Remark \ref{rem:isomat}, the notion of ``isolation'' can also be generalized algebraically as follows. 
For a given skew polynomial algebra $S_{\alpha}$ and $i \in [n]$, 
there exists a unique skew polynomial algebra $\widetilde{S_{\alpha}}$
such that $\GrMod \widetilde{S_{\alpha}}$ is equivalent to $\GrMod S_{\alpha}$ and $x_i$ is a central element of $\widetilde{S_{\alpha}}$. 
We can also regard a subset $T \subset [n]$ as an ``independent set'' if $x_i$ and $x_j$ are commutative for any $i,j \in T$.
Then we can apply similar proofs for Propositions~\ref{prop:isol1} and \ref{prop:isol2}, so the same statements also hold for the point simplicial complex $\Delta_{\alpha}$ for a skew polynomial algebra $S_{\alpha}$. 
\end{rem}

\section*{Acknowledgments}
The authors would like to thank Koichiro Tani for performing the computational experiments instead of them.
The authors would also like to thank the anonymous referee for carefully reading the paper and suggesting many improvements. 
The first author was supported by supported by JSPS KAKENHI Grant Number JP24K00521. 
The second author was supported by JSPS KAKENHI Grant Number JP22K03222.

\end{document}